\documentclass[11pt,reqno]{amsart}
\usepackage{graphicx}

\usepackage{pstricks}
\usepackage[all]{xy}
\usepackage{hyperref}
\usepackage{graphicx}
\usepackage[latin9]{inputenc}
\usepackage{amsfonts,amssymb}
\usepackage{mathrsfs}
\usepackage{amssymb}

\advance\hoffset by -0.5 truecm

\usepackage{color} 

\definecolor{green}{rgb}{0,0.6,0}
\definecolor{blue}{rgb}{0,0,1}

\hypersetup{colorlinks,
linkcolor=blue,
filecolor=green,
citecolor=green}

\theoremstyle{plain}

\newtheorem{neu}{}[section]

\newtheorem*{Cor*}{Corollary}
\newtheorem{Thm}[neu]{Theorem}
\newtheorem*{Thm*}{Theorem}
\newtheorem{Prop}[neu]{Proposition}
\newtheorem*{Prop*}{Proposition}
\theoremstyle{definition}
\newtheorem{Lemma}[neu]{Lemma}
\newtheorem*{Rmk*}{Remark}
\newtheorem{Rmk}[neu]{Remark}

\newtheorem*{Ex*}{Example}

\newtheorem*{Qu*}{Question}

\newtheorem{Def}[neu]{Definition}

\theoremstyle{remark}

\theoremstyle{definition}

\newcommand{\p}{\partial}
\newcommand{\om}{\omega}

\newcommand{\pf}{\longrightarrow}

\newcommand{\F}{{\mathbb{F}}}
\newcommand{\N}{{\mathbb{N}}}
\newcommand{\Z}{{\mathbb{Z}}}
\newcommand{\R}{{\mathbb{R}}}

\newcommand{\Q}{{\mathbb{Q}}}

\newcommand{\FF}{{\mathrm{F}}}

\newcommand{\LLL}{\mathscr{L}}

\renewcommand{\H}{\mathrm{H}}

\newcommand{\id}{\mathrm{id}}
\newcommand{\st}{\mathrm{st}}
\newcommand{\Id}{\mathrm{Id}}

\newcommand{\CF}{\mathrm{CF}}
\newcommand{\HF}{\mathrm{HF}}
\newcommand{\RFH}{\mathrm{RFH}}

\newcommand{\CM}{\mathrm{CM}}
\newcommand{\HM}{\mathrm{HM}}

\newcommand{\Crit}{{\rm Crit}}

\renewcommand{\AA}{\mathcal{A}}
\newcommand{\HH}{\mathcal{H}}
\newcommand{\DD}{\mathcal{D}}

\newcommand{\BB}{\mathcal{B}}
\newcommand{\VV}{\mathcal{V}}

\newcommand{\CC}{\mathscr{C}}
\newcommand{\MM}{\mathscr{M}}
\newcommand{\x}{\times}

\newcommand{\beq}{\begin{equation}}
\newcommand{\beqn}{\begin{equation}\nonumber}
\newcommand{\eeq}{\end{equation}}

\newcommand{\bea}{\begin{equation}\begin{aligned}}
\newcommand{\bean}{\begin{equation}\begin{aligned}\nonumber}
\newcommand{\eea}{\end{aligned}\end{equation}}

\numberwithin{equation}{section}
\numberwithin{figure}{section}

\begin{document}
\title[Invariance property of Morse homology on noncompact manifolds]{Invariance property of Morse homology \\on noncompact manifolds}
\author{Jungsoo Kang}

\address{Department of Mathematical Sciences, Seoul National University (SNU),
Kwanakgu Shinrim, San56-1 Seoul, South Korea, Email:
hoho159@snu.ac.kr}

\begin{abstract}
In this article, we focus on the invariance property of Morse homology on noncompact manifolds. We expect to apply outcomes of this article to several types of Floer homology, thus we define Morse homology purely axiomatically and algebraically. The Morse homology on noncompact manifolds generally depends on the choice of Morse functions; it is easy to see that critical points may escape along homotopies of Morse functions on noncompact manifolds. Even worse, homology classes also can escape along homotopies even though critical points are alive. The aim of the article is two fold. First, we give an example which breaks the invariance property by the escape of homology classes and find appropriate growth conditions on homotopies which prevent such an escape. This takes advantage of the bifurcation method. Another goal is to apply the first results to the invariance problem of Rabinowitz Floer homology. The bifurcation method for Rabinowitz Floer homology, however, is not worked out yet. Thus believing that the bifurcation method is applicable to Rabinowitz Floer homology, we study the invariance problems of Rabinowitz Floer homology.
\end{abstract}

\maketitle
\tableofcontents

\section{Introduction}
In recent times, several types of Morse and Floer homology have been developed and widely studied. The power of Morse and Floer homologies is the invariance property; that is, these homologies are independent of the choice of the Morse or Hamiltonian functions (or symplectic forms). Unfortunately this is rarely true on noncompact manifolds. One can easily find two Morse functions on $\R$ such that the respective Morse homologies are not isomorphic. There are two methods to show the invariance property of Morse and Floer theory. The first one is the continuation method; we count gradient flow lines of a homotopy between two Morse functions and this gives a continuation homomorphism between two respective Morse homologies. The other tool is the bifurcation method; we again consider a homotopy between two Morse functions and analyze how the Morse chain varies along the homotopy. It was introduced by Floer \cite{Fl1} to show the invariance of Lagrangian Floer homology though it was not completely justified, but this method was replaced by the continuation argument by himself in \cite{Fl2}. Recently, Hutchings and Lee \cite{Hu,Lee1,Lee2} completed the analysis required in the bifurcation method and it was used in \cite{Co,Us}; to be more specific, Hutchings worked on generalized Morse theory and Lee worked on Floer theories for the torsion invariant of Morse and Floer theories \cite{HL1,HL2}. In particular, Lee proved in the Floer theoretic setting that there exists a ``regular homotopy of Floer systems (RHFS)'' such that only two types of degeneracies can happen along this homotopy, namely ``birth-death'' and ``handle-slide''.\\[-2ex]

Both methods are painful but useful in the following sense. For the continuation method, we need compactness for gradient flow lines of a time-dependent action functional; but it gives a concrete isomorphism. On the other hand, we need to study gluing and decaying for the bifurcation method. However once this required analysis is worked out, it enables the detection of more general things; for example, invariance of the Reidemeister torsion in Morse and Floer theories has been studied in depth by Hutchings and Lee \cite{Hu,HL1,HL2,Lee1,Lee2} using the bifurcation method.\\[-2ex]

The purpose of this article is two fold. First, we investigate the invariance problem of Morse homology on noncompact manifolds by using the bifurcation method. As we have already mentioned, Morse homology can change along homotopies on noncompact manifolds. This incident can obviously be caused by the escape of critical points, see Remark \ref{Rmk:escape of a critical point} and Figure \ref{fig:escape}; even worse, homology classes also can escape as described in Theorem A. This is a very surprising phenomenon because homology classes escape to infinity whereas critical points keep alive. How does this happen? Let us change this problem to an interesting story. Suppose that there is no bus to go to heaven, how can we reach heaven? The answer is to transfer infinitely many buses of higher and higher speed and then we eventually arrive in heaven in finite time although no buses arrives at heaven. Using this idea, we will illustrate that homology classes can escape to infinity by infinitely many handle-slides or birth-deaths. It shows that if there are infinitely many generators of chain groups, a homology class may disappear even though generators may not. In the classical Floer theory, chain groups of Floer homology are of finite dimension over a suitable Novikov ring. However it is not true anymore for Rabinowitz Floer homology. So the escape of homology classes is a new phenomenon arising in Rabinowitz Floer theory. In order to prove the invariance property of Rabinowitz Floer homology, Cieliebak-Frauenfelder-Paternain \cite{CFP} and Bae-Frauenfelder \cite{BF} took advantage of the continuation method. On the other hand, one may expect that the invariance of Rabinowitz Floer homology can be proved by means of the bifurcation method.\\[-2ex]

\noindent\textbf{Question A}. Is the bifurcation method applicable to Rabinowitz Floer theory?\\[-2ex]

We expect that the above question will be answered in the near future.  The second aim of the article is to apply the first results to the invariance problem of Rabinowitz Floer homology. Since the bifurcation method for Rabinowitz Floer homology, however, is not worked out yet. Thus believing the bifurcation method is applicable to Rabinowitz Floer theory, we study the invariance problems of Rabinowitz Floer homology. The following Question B is our starting point.\\[-2ex]

\noindent\textbf{Question B}. Believing the answer to Question A is positive, can we prove the invariance property of Rabinowitz Floer homology using the bifurcation method?\\[-2ex]

In Theorem B and C, we give sufficient conditions preventing the escape of homology classes; more precisely, in Theorem B we impose an appropriate growth restriction on homotopies so that Morse homology is invariant; moreover, we show that a given homology class never escapes under a mild growth restriction in Theorem C. We apply these results to Rabinowitz Floer homology developed by Cieliebak-Frauenfelder \cite{CF}. The invariance problem of Rabinowitz Floer homology on stable or contact manifolds is not completely known yet. Interestingly, if the answer to Question A is positive then by examining the bifurcation process, we can prove the invariance property of Rabinowitz Floer homology along stable tame homotopies, see section 4; furthermore we are able to slightly relax the tameness condition.

\subsubsection*{Acknowledgments} I am indebted to Urs Frauenfelder for various suggestions and discussions.

\section{Cerf diagram and Morse homology}
In this section, we define Morse homology purely axiomatically and algebraically because we hope our results can be applied to all various type of Floer theories which satisfy the basic ingredients of Morse homology theory. Though our story begins with algebraic axioms, the classical Morse and Floer homologies satisfy these axioms.
\subsection{Cerf tuple and Cerf diagram}
We set projection maps
$$\pi_1,\, \pi_3:\R^2\x[0,1]=\R\x[0,1]\x\R\pf\R,\quad \pi_2:\R\x[0,1]\x\R\pf[0,1].$$
\begin{Def}
We call a tuple $\CC=(C,F)$ a {\em Cerf tuple} if the following conditions hold.
\begin{itemize}
\item[$(\CC1)$] $C$ is a one dimensional manifold with boundary such that each connected component of $C$ is compact.
\item[$(\CC2)$] $F:C\pf\R^2\x[0,1]$ is a smooth map with the property: $\pi_3\circ F$ is proper and, for a connected component $c\subset C$, $F|_{c}:c\pf\R^2\x[0,1]$ is either a Legendrian knot or a Legendrian chord which begins and ends on the pre-Lagrandian submanifold $\R^2\x\{0\}$ or $\R^2\x\{1\}$.
\end{itemize}
We refer to the appendix for the notions of the Legendrian knot and chord and the pre-Lagangian. We denote by $F_i:=\pi_i\circ F$, $i=1,2,3$. For a given Cerf tuple, the front projection of parameterized Legendrian curves $F(C)$ is called the {\em Cerf diagram}:
$$\big\{(F_2(c),F_3(c))\,|\,c\subset C\big\}\subset[0,1]\x\R.$$
\end{Def}

\begin{figure}[htb]\label{Cerf-tuple}
\input{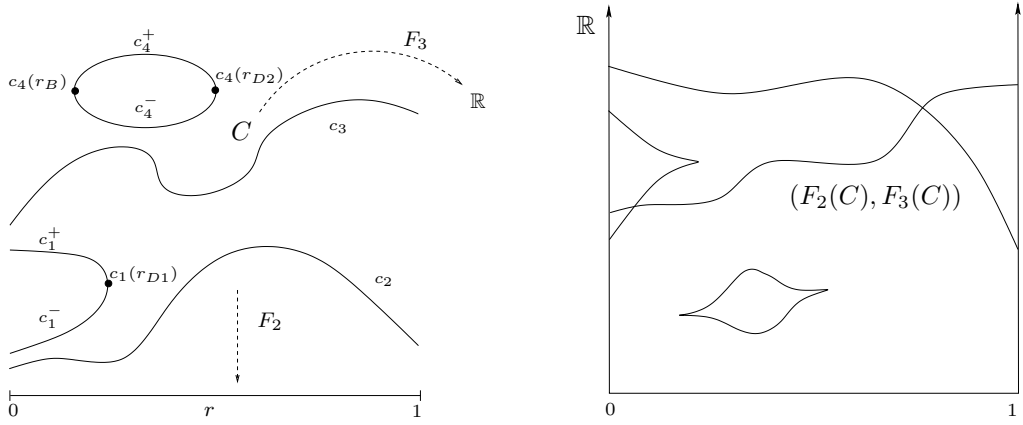}\caption{Cerf tuple and Cerf diagram}
\end{figure}

\begin{Rmk}
Let us take a look at the Cerf triple and the Cerf diagram in the Morse theoretic viewpoint. We have a one-parameter Morse functions $\{f_r\}_{r\in[0,1]}$ on a manifold $M$. Then a one dimensional manifold $C$ corresponds to $\Crit f_r\subset M\x[0,1]$ and the smooth function $F_2$ on $C$ indicates the parameter $r$ and $F_3$ is nothing but the Morse function $f_r$ at critical points.
\end{Rmk}
\begin{Rmk}
As can be seen in Figure \ref{Cerf-tuple}, cusps appear in the Cerf diagram. About the reason, we refer to Remark \ref{rmk:legendrian lifting}.
\end{Rmk}

\noindent{\em Degeneracies}. For clarity, we indicate dependence of the parameter $r\in[0,1]$ by $c_1(r)\in c_1$ for $F_2(c_1(r))=r$. If $c_1$ has two points with same $F_2$-value, we denote by $c_1^+(r)$ and $c_1^-(r)$. We often write the subscripts $D$, $B$, and $H$ to allude the degenerate types, namely birth-deaths or handle-slides. We also define the {\em set of deaths} and the {\em set of birthes} as follows:
\bean
\bullet&\quad\DD_0:=\bigr\{c(r_D)\in c\subset C\,\bigr|\,\textrm{$c(r_D)$ is a local maximum point of $F_2$}\bigr\},\\[0.5ex]
\bullet&\quad\BB_0:=\bigr\{c(r_B)\in c\subset C\,\bigr|\,\textrm{$c(r_B)$ is a local minimum point of $F_2$}\bigr\}.
\eea
We note that the above sets are discrete in $C$. For $c(r_B)\in\BB$, we note that for a small $\epsilon>0$ then $F_2^{-1}(r_B+\epsilon)|_c$ consists of two distinct points. As mentioned, we denote each of them by $c_B^+(r_B+\epsilon)$ and $c_B^-(r_B+\epsilon)$. We analogously define $c_D^+(r_D-\epsilon)$ and $C_D^-(r_D-\epsilon)$ for deaths. These degeneracies, birth-deaths, is caused by the Cerf tuple itself.

\subsection{Graph structure}
\begin{Def}
A {\em graph structure} on a topological space $G$ is a discrete subset $V$ of $G$ such that $G\setminus V$ is a 1-dimensional manifold. A pair $(G,V)$ is called a {\em graph}. An element in $V$ is called a {\em vertex} and each connected component of $G\setminus V$ is called an {\em edge}.
\end{Def}
\begin{Def}
A graph structure $V_0$ on $G$ is called a {\em supergraph structure} of the graph $(G,V)$ if $V_0\supseteq V$.
\end{Def}
Let $(G,V)$ be a graph. For index sets $I$ and $J$, we set
$$
V=\{v_i\,|\,v_i\in G,\,i\in I\},\qquad \pi_0(G\setminus V)=\{e_j\,|\,e_j\subset G,\,j\in J\}.
$$
Let $\F$ be any principal ideal domain (e.g. $\Z_2$, $\Z$, or $\Q$) with the discrete topology.
\begin{Def}
A function $\phi:G\pf\F$ is called a {\em step function} on $(G,V)$ if it can be written as
$$
\phi(g)=\sum_{i\in I}f_i\chi_{v_i}(g)+\sum_{j\in J}f_j\chi_{e_j}(g),\quad  g\in G,\quad f_i,\, f_j\in\F
$$
where $\chi_{v_i}$ and $\chi_{e_j}$ are the indicator functions defined by
\beqn
\chi_{v_i}(g)=\left\{\begin{aligned} 1 & \textrm{   if   } g=v_i\\ 0 & \textrm {   if   } g\neq v_i\end{aligned}\right. \qquad
\chi_{e_j}(g)=\left\{\begin{aligned} 1 & \textrm{   if   } g\in e_j\\ 0 & \textrm {   if   } g\notin  e_j\end{aligned}\right.
\eeq
\end{Def}
We define the {\em fiber product} of $F_2:C\pf[0,1]$ as follows:
$$
C\x_{F_2} C:=\big\{(c_1,c_2)(r)\in C\x C\,|\,F_2(c_1(r))=r=F_2(c_2(r))\big\}.
$$
This fiber product has a natural graph structure given by the Cerf tuple. The sets $\DD_0$ and $\BB_0$ defined in the previous subsection give the following subsets of the fiber product of $F_2$.
\bean
\bullet&\quad\DD:=\bigr\{(c_1,c_2)(r_D)\in C\x_{F_2} C\,\bigr|\,c_1(r_D)\textrm{ or } c_2(r_D)\in\DD_0\bigr\},\\[0.5ex]
\bullet&\quad\BB:=\bigr\{(c_1,c_2)(r_B)\in C\x_{F_2} C\,\bigr|\,c_1(r_B)\textrm{ or } c_2(r_B)\in\BB_0\bigr\}.
\eea
In particular, we define the diagonals of $\DD$ and $\BB$ as follows:
\bean
\bullet&\quad\triangle\DD:=\bigr\{(c,c)(r_D)\in C\x_{F_2} C\,\bigr|\,c(r_D)\in\DD_0\bigr\},\\[0.5ex]
\bullet&\quad\triangle\BB:=\bigr\{(c,c)(r_B)\in C\x_{F_2} C\,\bigr|\,c(r_B)\in\BB_0\bigr\}.
\eea
It is easy to see that $C\x_{F_2} C$ is topologically nothing but an union of closed intervals and wedge sums of closed intervals where points in $\triangle\DD\cup\triangle\BB$ are identified. Furthermore the following discrete set $\VV_\CC$ induced by the cerf tuple $\CC$ endows a natural graph structure on $C\x_{F_2}C$.
$$
\VV_\CC:=\DD\cup\BB\subset C\x_{F_2} C.
$$
\begin{figure}[htb]\label{graph_str}
\input{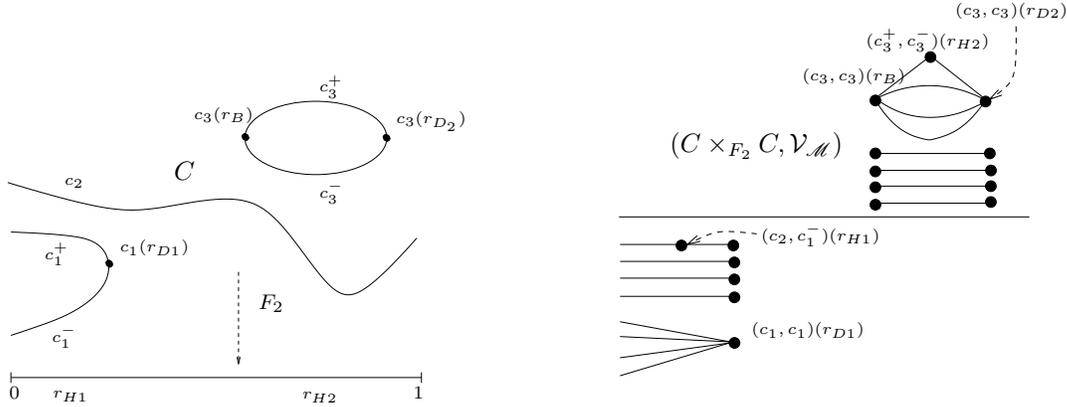}\caption{Graph structure}
\end{figure}

\subsection{Axioms on $\gamma$}

\begin{Def}
A function
$$
\gamma:C\x_{F_2} C\pf\F
$$
is called the {\em flow line counter} if it is a step function on a supergraph $(C\x_{F_2}C,\VV_{\mathscr{M}})$ of $(C\x_{F_2}C,\VV_\CC)$ such that the following holds. There exists function $\delta:\HH:=\VV_{\mathscr{M}}\setminus\VV_\CC\pf\F$ such that $\gamma$ together with $\delta$ satisfy the following five axioms $(\gamma1)-(\gamma5)$. The set $\HH$ is called the {\em set of handle-slides} and the function $\delta$ is called the {\em jump function}. We set $\gamma=0$ for convention when $\gamma$ has the infinite value in $\F$.
\end{Def}
We write $\gamma_r(c_1,c_2)=\gamma(c_1,c_2)(r)$ and  $\delta_{r_H}(c_H^+,c_H^-)=\delta(c_H^+,c_H^-)(r_H)$ for brevity.
\begin{Rmk}
As can be seen in $(\gamma3)$, the jump function $\delta$ measures the discontinuity of $\gamma$ at $\HH$.
\end{Rmk}

The important data is the value of $\gamma$ on edges of $C\x_{F_2} C$. $\gamma$ is constant on each edges, but the value may jump at $\VV_{\mathscr{M}}=\DD\cup\BB\cup\HH$. Thus we need to examine how the value of $\gamma$ changes at $\VV_{\mathscr{M}}$. For such a reason, we define the approximated value of $\gamma$ to compare the value of $\gamma$ before and after $\Lambda$ defined by
$$
\Lambda:=\{r\in [0,1]\,|\, (c_1,c_2)(r)\in\VV_\MM\}.
$$
In particular we indicates the type of degeneracies of {\em degenerate points} in $\Lambda$ as below:
\bean
&\bullet\quad F_2(\DD):=\{r_D\in[0,1]\,|\,(c_1,c_2)(r_D)\in\DD\},\\[0.5ex]
&\bullet\quad F_2(\BB):=\{r_B\in[0,1]\,|\,(c_1,c_2)(r_B)\in\BB\},\\[0.5ex]
&\bullet\quad F_2(\HH):=\{r_H\in[0,1]\,|\,(c_1,c_2)(r_H)\in\HH\}.
\eea
We assume that those points are \textbf{disjoint} in $[0,1]$. By definition,
$$
\Lambda=F_2(\DD)\cup F_2(\BB)\cup F_2(\HH)\subset[0,1].\\
$$

\begin{Def}
$(c_1,c_2)(r-\epsilon)\in C\x_\pi C$ is the {\em left approximation} of $(c_1,c_2)(r)\in C\x_\pi C$ if
\begin{itemize}
\item $\epsilon>0$,
\item $\lim_{\epsilon\to0}(c_1,c_2)(r-\epsilon)=(c_1,c_2)(r).$
\end{itemize}
Then we define for the left approximation of $(c_1,c_2)(r)$,
$$
\gamma_{r}^-(c_1,c_2):=\lim_{\epsilon\to0}\gamma_{r-\epsilon}(c_1,c_2).
$$
Since the non-degenerate function $\gamma$ is constant on each edges, $\gamma^-_r(c_1,c_2)$ is well-defined.
Analogously, we also define $(c_1,c_2)(r+\epsilon)$ the {\em right approximation} of $(c_1,c_2)(r)\in C\x_\pi C$ together with $\gamma_r^+(c_1,c_2)\in\F$.
\end{Def}

The non-degenerate function $\gamma$ satisfies the following five axioms.
\begin{itemize}
\item[($\gamma1$)] For $r\in[0,1]\setminus\Lambda$, $\gamma_r(c_1,c_2)=0$ if $F_3(c_2)\geq F_3(c_1)$;\\[1ex]
for $(c_1,c_2)(r_H)\in\HH$, $\delta_{r_H}(c_1,c_2)=0$ if $F_3(c_2)\geq F_3(c_1)$.\\
\item[($\gamma2$)] For $(c_1,c_3)(r)\in C\x_{F_2}C$, $r\in[0,1]\setminus\Lambda$,
$$
\sum_{c_2\subset C}\gamma_r(c_1,c_2)\gamma_r(c_2,c_3)=0.
$$
\item[($\gamma3$)] At $r_H\in F_2(\HH)$, the following holds.
$$
\gamma_{r_H}^+(c_1,c_3)=\gamma^-_{r_H}(c_1,c_3)+\!\!\!\!\!\!\!\sum_{\substack{c_2\subset C;\\(c_1,c_2)(r_H)\in\HH}}\!\!\!\!\!\!\!\delta_{r_H}(c_1,c_2)\gamma^-_{r_H}(c_2,c_3)
-\!\!\!\!\!\!\!\sum_{\substack{c_2\subset C;\\(c_2,c_3)(r_H)\in\HH}}\!\!\!\!\!\!\!\delta_{r_H}(c_2,c_3)\gamma^+_{r_H}(c_1,c_2).
$$
\item[($\gamma4$)] At $r_B\in F_2(\BB)$, $\gamma_{r_B}^+(c_B^+,c_B^-)$ is an invertible element in $\F$ and the following holds.
\bean
\gamma_{r_B}^+(c_1,c_2)&=\gamma^-_{r_B}(c_1,c_2)+\gamma^+_{r_B}(c_1,c_B^-)\gamma^+_{r_B}(c_B^+,c_B^-)^{-1}\gamma^+_{r_B}(c_B^+,c_2);\\
\gamma_{r_B}^+(c_1,c_B^+)&=\gamma_{r_B}^+(c_B^-,c_1)=\gamma_{r_B}^+(c_B^+,c_2)=0,\quad c_2\neq c_B^-.
\eea
\item[($\gamma5$)] At $r_D\in F_2(\DD)$, $\gamma_{r_D}^-(c_D^+,c_D^-)$ is an invertible element in $\F$ and the following holds.
\bean
\gamma_{r_D}^+(c_1,c_2)&=\gamma^-_{r_D}(c_1,c_2)-\gamma^-_{r_D}(c_1,c_D^-)\gamma^-_{r_D}(c_D^+,c_D^-)^{-1}\gamma^-_{r_D}(c_D^+,c_2);\\
\gamma_{r_D}^-(c_1,c_D^+)&=\gamma_{r_D}^-(c_D^-,c_1)=\gamma_{r_D}^-(c_D^+,c_2)=0,\quad c_2\neq c_D^-.
\eea
\end{itemize}
\begin{Rmk}
$(\gamma1)$ guarantees that $\gamma_r(c,c)=0$, $\gamma^+_{r_B}(c_B^-,c_B^+)=0$, and $\gamma^-_{r_D}(c_D^-,c_D^+)=0$. By the properness of $F_3$, the formula in $(\gamma2)$ is a finite sum, i.e. for fixed $(c_1,c_3)(r)\in C\x_{F_2} C$, there are only finitely many $c_2\in F_2^{-1}(r)$ such that $\gamma_r(c_1,c_2)\gamma_r(c_2,c_3)$ is nonzero.
\end{Rmk}
\begin{Rmk}
In the Morse theoretic framework, the axioms and functions can be interpreted as the following.
\begin{itemize}
\item $\delta$ counts the number of degenerate gradient flow lines between critical points with same indices (non-generic phenomenon).
\item $\gamma$ counts the number of gradient flow lines.
\item $(\gamma1)$ implies that the action value decreases along gradient flow lines.
\item $(\gamma2)$ means that the function $\gamma$ gives the boundary operator of the Morse chain complex by counting gradient flow lines.
\item $(\gamma3)-(\gamma5)$ signify how the value of $\gamma$ (or the boundary operator) changes at degenerate points in $\Lambda$, see Figure \ref{fig:Handle-slide} and Figure \ref{fig:birth}.
\end{itemize}
\end{Rmk}

\subsection{Morse homology}
\begin{Def}
We call a pair $\mathscr{M}=(\CC,\gamma)$ the {\em Morse tuple} which consists of a Cerf tuple $\CC=(C,F)$ and a flow line counter $\gamma$ on $\CC$.
\end{Def}
First of all, we define the following set which is finite by the properness of $F_3$. For $a\leq b\in\R$, $r\in[0,1]\setminus\Lambda$,
$$
C^{(a,b)}(\MM,r):=\big\{c(r)\in C\,\big|\, F_3(c(r))\in (a,b)\big\}.
$$
Then we have the following $\F$-module by tensoring $\F$.
$$
\CM^{(a,b)}(\MM,r):=C^{(a,b)}(\MM,r)\otimes\F.
$$
Next, we define a boundary operator $\p$ using $\gamma$.
\bean
\p_r^{(a,b)}:\CM^{(a,b)}(\MM,r)&\pf \CM^{(a,b)}(\MM,r)\\
c_1(r)&\longmapsto\sum_{c_2\subset C}\gamma_r(c_1,c_2)\cdot c_2(r).
\eea
Recall that we set $\gamma_r(c_1,c_2)=0$ if it equals to infinity. We note that $(\CM^{(a,b)}(\MM,r),\p_r^{(a,b)})$ is indeed a chain complex due to Axiom $(\gamma2)$; therefore, we get {\em filtered Morse homology}:
$$
\HM^{(a,b)}(\MM,r):=\H\big(\CM^{(a,b)}(\MM,r),\p_r^{(a,b)}\big),\qquad r\in[0,1]\setminus\Lambda
$$
and then taking direct and inverse limits, we obtain {\em (full) Morse homology}:
$$
\HM(\MM,r):=\lim_{\substack{\pf\\b\to\infty}}\lim_{\substack{\longleftarrow\\a\to-\infty}}\HM^{(a,b)}(\MM,r),\qquad r\in[0,1]\setminus\Lambda.
$$

\subsection{Invariance}
Thanks to the fact that $\gamma$ is constant at each edges, we easily derive the invariance property of Morse homology on a non-degenerate interval.
$$
\HM(\MM,r_1)\cong\HM(\MM,r_2)\quad\textrm{whenever }\,\,[r_1,r_2]\cap\Lambda=\emptyset.
$$

In next three propositions, we shall show that Morse homology is unchanged even after a handle-slide and a birth-death. Due to the axioms $(\gamma_3)$, $(\gamma_4)$, and $(\gamma_5)$ together with $(\CC2)$, we know that how Morse chain and the boundary operator vary by passing through those degenerate points. In fact, Lee \cite{Lee1,Lee2} completes all the necessary analysis of the bifurcation method in Floer theory argued originally by Floer \cite{Fl1}; accordingly she proved that all axioms and hypotheses of this article hold in Floer theory, but she did not explicitly prove the invariance property even though it immediately follows. Instead, she concerned with the torsion invariants in Floer theory, (see the introduction). Usher \cite{Us} stated and proved the invariance property described below.

\begin{Prop}\cite{Lee1,Lee2,Us}
If $[r_0,r_1]\cap\Lambda=\{r_H\in F_2(\HH)\}$, $\HM(\MM,r_0)\cong\HM(\MM,r_1)$.
\end{Prop}
\begin{proof}
We choose continuous functions $a(r),\,b(r):[0,1]\pf\R$ such that the images of $a$ and $b$ do not intersect with the Cerf diagram.
We set the map
\bea\label{eq:chain map at HS}
A:\CM^{(a(r_1),b(r_1))}(\MM,r_1)&\pf \CM^{(a(r_0),b(r_0))}(\MM,r_0)\\
c(r_1)&\longmapsto c(r_0)+\!\!\!\sum_{\substack{c_H^-\subset C;\\(c,c_H^-)(r_H)\in\HH}}\!\!\!\delta_{r_H}(c,c_H^-)\cdot c_H^-(r_0).
\eea
Since $A$ is invertible, it suffices to show that $A$ is a chain map then it gives an isomorphism on the homology level. We abbreviate $\p_{r_0}^{(a(r_0),b(r_0))}$ resp. $\p_{r_1}^{(a(r_1),b(r_1))}$ by $\p_-$ resp. $\p_+$.\\[-2ex]

\noindent\underline{Claim}: $A$ is a chain map, i.e. $A\circ\p_+=\p_-\circ A$.\\[-2ex]

\noindent{Proof of the claim}. We compute for $c(r_1)\in C^{(a(r_1),b(r_1))}(\MM,r_1)$,
\bean
A\circ\p_+&-\p_-\circ A(c(r_1))\\
&=A\Big(\sum_{c'\subset C}\gamma_{r_1}(c,c')\cdot c'(r_1)\Big)-\p_-\Big(c(r_0)+\sum_{c_H^-\subset C}\delta_{r_H}(c,c_H^-)\cdot c_H^-(r_0)\Big)\\
&=\sum_{c'\subset C}\gamma_{r_1}(c,c')\cdot\Big(c'(r_0)+\sum_{c_H^-\subset C}\delta_{r_H}(c',c_H^-)\cdot c_H^-(r_0)\Big)\\
&\quad-\sum_{c'\subset C}\Big(\gamma_{r_0}(c,c')-\sum_{c_H^-\subset C}\delta_{r_H}(c,c_H^-)\gamma_{r_0}(c_H^-,c')\Big)\cdot c'(r_0)\\
&=0
\eea
The last equality follows from the axiom $(\gamma3)$ and this computation finishes the proof of the claim, hence the proposition.
\end{proof}

\begin{Prop}\cite{Lee1,Lee2,Us}
If $[r_0,r_1]\cap\Lambda=\{r_B\in F_2(\BB)\}$, $\HM(\MM,r_0)\cong\HM(\MM,r_1)$.
\end{Prop}
\begin{proof}
We choose again continuous functions $a(r),\,b(r):[0,1]\pf\R$ such that the images of $a$ and $b$ do not intersect with the Cerf diagram. We abbreviate $C^{(a(r_0),b(r_0))}(\MM,r_0)$ resp. $C^{(a(r_1),b(r_1))}(\MM,r_1)$ by $C_-$ resp. $C_+$. We note that there exists the natural bijection
\bean
\Psi:C_-\cup\{c^+_B(r_1),c^-_B(r_1)\}&\pf C_+\\[1ex]
c(r_0)\in C_-&\longmapsto c(r_1)\\
c^\pm_B(r_1)&\longmapsto c^\pm_B(r_1)
\eea
and it gives the isomorphism
$$
\CM^{(a(r_1),b(r_1))}(\MM,r_1)\cong\CM^{(a(r_0),b(r_0))}(\MM,r_0)\oplus\F\langle c^+_B(r_1),c^-_B(r_1)\rangle.
$$
For convenience we identify $C_-$ with $\Psi(C_-)\subset C_+$; but one can easily distinguish elements in $C_-$ or $\Psi(C_-)$ by the parameters $r_0$ or $r_1$. We set the chain maps:
\bean
i:\CM^{(a(r_0),b(r_0))}(\MM,r_0)&\pf\CM^{(a(r_1),b(r_1))}(\MM,r_1)\\
c(r_0)&\longmapsto c(r_1)-\gamma_{r_1}(c,c^-_B)\gamma_{r_B}^+(c^+_B,c^-_B)^{-1}\cdot c^+_B(r_1)\\[1ex]
p:\CM^{(a(r_1),b(r_1))}(\MM,r_1)&\pf\CM^{(a(r_0),b(r_0))}(\MM,r_0)\\
c(r_1)\in C_-&\longmapsto c(r_0)\\
c^+_B(r_1)&\longmapsto 0\\
c^-_B(r_1)&\longmapsto -\sum_{c\subset C_-}\gamma_{r_B}^+(c^+_B,c^-_B)^{-1}\gamma_{r_1}(c^+_B,c)\cdot c(r_0)
\eea
where $\gamma_{r_B}^+(c_B^+,c_B^-)^{-1}$ is the inverse of $\gamma_{r_B}^+(c_B^+,c_B^-)$ in $\F$.\\[-2ex]

\noindent\underline{Claim 1}: $i$ and $p$ are indeed chain maps, namely $\p_+ \circ i=i\circ\p_-$ and $\p_-\circ p=p\circ\p_+$.\\[-2ex]

\noindent{Proof of Claim 1}. Using the axioms $(\gamma_2)$ and $(\gamma4)$, we compute that for any $c(r_0)\in C_-$,
\bea\label{eq:vanishing eq1}
\sum_{c'\subset C_-}\gamma_{r_0}(c,c')\gamma_{r_1}(c',c_B^-)&=\sum_{c'\subset C_-}\gamma_{r_1}(c,c')\gamma_{r_1}(c',c_B^-)\\
&\quad-\sum_{c'\subset C_-}\gamma_{r_1}(c,c_B^-)\gamma_{r_B}^+(c_B,c_B^-)\gamma_{r_1}(c_B^+,c')\gamma_{r_1}(c',c_B^-)\\
&=0.
\eea
Similarly, we also can show that for $c(r_1)\in C_-$,
\beq\label{eq:vanishing eq2}
\sum_{c'\subset C_-}\gamma_{r_1}(c_B^+,c')\gamma_{r_0}(c',c)=0.
\eeq
With the axioms $(\gamma_2)$ and $(\gamma4)$ again, we calculate for $c(r_0)\in C_-$,
\bean
\bullet\quad & i\,\circ\,  \p_--\p_+\circ i\,(c(r_0))\\
&=i\bigr(\sum_{c'\subset C_-}\gamma_{r_0}(c,c')\cdot c'(r_0)\bigr)-\p_+\bigr(c(r_1)+\gamma_{r_1}(c,c_B^-)\gamma_{r_B}^+(c_B^+,c_B^-)^{-1}\cdot c_B^+(r_1)\bigr)\\
&=\sum_{c'\subset C_-}\gamma_{r_0}(c,c')\cdot c'(r_1)+\underbrace{\sum_{c'\subset C_-}\gamma_{r_0}(c,c')\gamma_{r_1}(c',c_B^-)}_{=0 \textrm{ by }\eqref{eq:vanishing eq1}}\gamma_{r_B}^+(c_B^+,c_B^-)^{-1}\cdot c_B^+(r_1)\\
&\quad-\sum_{c'\subset C_+}\gamma_{r_1}(c,c')\cdot c'(r_1)+\sum_{c'\subset C_+}\gamma_{r_1}(c,c_B^-)\gamma_{r_B}^+(c_B^+,c_B^-)^{-1}\gamma_{r_1}(c_B^+,c')\cdot c'(r_1)\\
&=\sum_{c'\subset C_-}\gamma_{r_0}(c,c')\cdot c'(r_1)-\sum_{c'\subset C_-}\gamma_{r_1}(c,c')\cdot c'(r_1)+\gamma_{r_1}(c,c_B^-)\cdot c_B^-(r_1)\\
&\quad +\sum_{c'\subset C_+}\gamma_{r_1}(c,c_B^-)\gamma_{r_B}^+(c_B^+,c_B^-)^{-1}\gamma_{r_1}(c_B^+,c')\cdot c'(r_1)\\
&=\sum_{c'\subset C_-}\gamma_{r_0}(c,c')\cdot c'(r_1)-\sum_{c'\subset C_-}\big\{\gamma_{r_0}(c,c')-\gamma_{r_1}(c,c_B^-)\gamma_{r_B}^+(c_B^+,c_B^-)^{-1}\gamma_{r_1}(c_B^+,c')\big\}\cdot c'(r_1)\\
&\quad -\gamma_{r_1}(c,c_B^-)\cdot c_B^-(r_1)+\sum_{c'\subset C_-}\gamma_{r_1}(c,c_B^-)\gamma_{r_B}^+(c_B^+,c_B^-)^{-1}\gamma_{r_1}(c_B^+,c')\cdot c'(r_1)\\
&\quad+\gamma_{r_1}(c,c_B^-)\gamma_{r_B}^+(c^+_B,c^-_B)^{-1}\gamma_{r_B}^+(c^+_B,c^-_B)\cdot c^-_B(r_1)\\
&=0.
\eea
The fourth equality follows from $(\gamma4)$. Similarly, we show $p\circ\p_+=\p_-\circ p$ for $c(r_1)\in C_-$, $c_B^-(r_1)$, and $c_B^+(r_1)$.
{\setlength\arraycolsep{2pt}
\begin{eqnarray*}
&\bullet\quad p\circ\p_+(c(r_1))&=p\bigr(\sum_{c'\subset C_+}\gamma_{r_1}(c,c')\cdot c'(r_1)\bigr)\\
&\quad&=p\bigr(\sum_{c'\subset C_-}\gamma_{r_1}(c,c')\cdot c'(r_1)+\gamma_{r_1}(c,c_B^-)\cdot c_B^-(r_1)\bigr)\\
&\quad&=\sum_{c'\subset C_-}\gamma_{r_1}(c,c')\cdot c'(r_0)-\sum_{c'\subset C_-}\gamma_{r_B}^+(c_B^+,c_B^-)^{-1}\gamma_{r_1}(c^+_B,c')\gamma_{r_1}(c,c_B^-)\cdot c'(r_0)\\
&\quad&=\sum_{c'\subset C_-}\gamma_{r_0}(c,c')\cdot c'(r_0)\\
&\quad&=\p_-(c(r_0))\\
&\quad&=\p_-\circ p(c(r_1)).\\[2ex]
\end{eqnarray*}
\begin{eqnarray*}
&\bullet\quad \p_-\circ p(c_B^-(r_1))&=\p_-\bigr(-\sum_{c'\subset C_-}\gamma_{r_B}^+(c_B^+,c_B^-)^{-1}\gamma_{r_1}(c_B^+,c')\cdot c'(r_0)\bigr)\\
&\quad&=-\gamma_{r_B}^+(c_B^+,c_B^-)^{-1}\sum_{c''\subset C_-}\underbrace{\sum_{c'\subset C_-}\gamma_{r_1}(c_B^+,c')\gamma_{r_0}(c',c'')}_{=0 \textrm{ by } \eqref{eq:vanishing eq2}} \cdot c''(r_0)\\
&\quad&=0\\
&\quad&=p\circ \p_+(c_B^-(r_1)).\\[2ex]
&\bullet\quad p\circ\p_+(c_B^+(r_1))&=p\bigr(\sum_{c'\subset C_+}\gamma_{r_1}(c_B^+,c')\cdot c'(r_1)\bigr)\\
&\quad&=p\bigr(\sum_{c''\subset C_-}\gamma_{r_1}(c_B^+,c'')\cdot c''(r_1)+\gamma_{r_B}^+(c_B^+,c_B^-)\cdot c_B^-(r_1)\bigr)\\
&\quad&=\sum_{c''\subset C_-}\gamma_{r_1}(c_B^+,c'')\cdot c''(r_0)-\sum_{c''\subset C_-}\gamma_{r_B}^+(c_B^+,c_B^-)^{-1}\gamma_{r_1}(c_B^+,c'')\gamma_{r_B}^+(c_B^+,c_B^-)\cdot c''(r_0)\\
&\quad&=0\\
&\quad&=\p_-\circ p(c_B^+(r_1)).
\end{eqnarray*}}

\noindent\underline{Claim 2}: $i$ and $p$ are chain homotopic.\\[-2ex]

\noindent{Proof of Claim 2}.
It obviously holds that $p\circ i=\id$. Thus it remains to show that $i\circ p\simeq\Id$, so we set the chain homotopy $D$ below.
{\setlength\arraycolsep{2pt}
\bea
D:\CM^{(a(r_1),b(r_1))}(\MM,r_1)&\pf\CM^{(a(r_1),b(r_1))}(\MM,r_1)\\
c(r_1),c_B^+(r_1)&\longmapsto 0\\
c_B^-(r_1)&\longmapsto \gamma_{r_B}^+(c_B^+,c_B^-)^{-1}\cdot c_B^+(r_1)
\eea
Then the following three simple calculations complete the proof of Claim 2 and hence the proposition. For $c(r_1)\in C_-$, and $c_B^\pm(r_1)$, we compute
\begin{eqnarray*}
&\bullet\quad\p^+\circ D+D\circ\p^+(c(r_1))&=0+D\bigr(\sum_{c'\subset C^-}\gamma_{r_1}(c,c')\cdot c'(r_1)+\gamma_{r_1}(c,c_B^-)\cdot c_B^-(r_1)\bigr)\\
&\quad&=\gamma_{r_1}(c,c_B^-)\gamma_{r_B}^+(c_B^+,c_B^-)^{-1}\cdot c_B^+(r_1)\\
&\quad&=c(r_1)-\bigr(c(r_1)-\gamma_{r_1}(c,c_B^-)\gamma_{r_B}^+(c_B^+,c_B^-)^{-1}\cdot c_B^+(r_1)\bigr)\\
&\quad&=\Id-i\circ p(c(r_1)).\\[2ex]
&\bullet\quad\p^+\circ D+D\circ\p^+(c_B^+(r_1))&=0+D\bigr(\sum_{c'\subset C_-}\gamma_{r_1}(c_B^+,c')\cdot c'(r_1)+\gamma_{r_B}^+(c_B^+,c_B^-)\cdot c_B^-(r_1)\bigr)\\
&\quad&=\gamma_{r_B}^+(c_B^+,c_B^-)\gamma_{r_B}^+(c_B^+,c_B^-)^{-1}\cdot c_B^+(r_1)\\
&\quad&=c_B^+(r_1)\\
&\quad&=\Id-i\circ p(c_B^+(r_1)).
\end{eqnarray*}
\begin{eqnarray*}
&\bullet\quad\p^+\circ D+D\circ\p^+(c_B^-(r_1))&=\gamma_{r_B}^+(c_B^+c_B^-)^{-1}\cdot c_B^+(r_1)+0\\
&\quad&=\sum_{c'\subset C_+}\gamma_{r_B}^+(c_B^+,c_B^-)^{-1}\gamma_{r_1}(c_B^+,c')\cdot c'(r_r)\\
&\quad&\quad+\gamma_{r_B}^+(c_B^+,c_B^-)^{-1}\gamma_+(c_B^+,c_B^-)\cdot c_B^-(r_1)\\
&\quad&=\sum_{c'\subset C_-}\gamma_{r_B}^+(c_B^+,c_B^-)^{-1}\gamma_{r_1}(c_B^+,c')\cdot c'(r_1)+c_B^-(r_1)\\
&\quad&=c_B^-(r_1)-\bigr(-\sum_{c'\subset C_-}\gamma_{r_B}^+(c_B^+,c_B^-)^{-1}\gamma_{r_1}(c_B^+,c')\cdot c'(r_1)\bigr)\\
&\quad&=\Id-i\circ p(c_B^-(r_1)).
\end{eqnarray*}}
\end{proof}

\begin{Prop}\cite{Lee1,Lee2,Us}
If $[r_0,r_1]\cap\Lambda=\{r_D\in F_2(\DD)\}$, $\HM(\MM,r_0)\cong\HM(\MM,r_1)$.
\end{Prop}
\begin{proof}
It follows immediately from the proof of the previous proposition by reversing arrows and signs.
\end{proof}

\section{Statement of the main results}
In the previous section, we showed that Morse homology is unchanged even if a handle-slide or a birth-death takes place. Unfortunately, however, this invariance property may not hold as passing through infinitely many degenerate points in $\Lambda$. The following shocking example describes that infinitely many transfer of the spectral value can make a homology class escape to infinity. As mentioned in the introduction, the following phenomenon only happen in Morse homology on noncompact manifolds and Rabinowitz Floer homology because Morse homology on compact manifolds or the classical Floer theory has finite dimensional chain groups over a suitable ring.\\[-2ex]

\noindent\textbf{Theorem A. (Escape of a homology class.)} {\em A homology class $h\in\HM(\MM,0)$ may be able to escape and thus there is no homology class in $\HM(\MM,1)$ corresponding to $h$.}

\begin{proof}
The only possible accident breaking the invariance property is the escape of a homology class since we have assumed that $C$ is compact; otherwise a critical point also can escape to infinity, see Remark \ref{Rmk:escape of a critical point}. It is caused by infinitely many ``transfer of the spectral value'' which occur by handle-slides and birth-deaths. For simplicity, we argue only with handle-slides.

For each homology class $h\in\HM(\MM,r)$, the spectral value is defined by
\beq\label{eq:spectral value}
\rho(h,r):=\inf_{\substack{\alpha\in\CM(\MM,r)\\ [\alpha]=h}}\sigma(\alpha,r)
\eeq
where
$$
\sigma(\alpha,r):=\sup_i\big\{F_3(c_i(r))\,\big|\,\alpha=\sum_i f_i\cdot c_i(r),\,\, 0\ne f_i\in\F,\,c_i(r)\in C\big\}.
$$
Moreover we set $\rho(0,r)=-\infty$ for convention.

We recall that even though $[r_1,r_2]\cap\Lambda=\{r_H\in F_2(\HH)\}$ we have the chain map \eqref{eq:chain map at HS}
\bean
A:\CM^{(a(r_2),b(r_2))}(\MM,r_2)&\pf \CM^{(a(r_1),b(r_1))}(\MM,r_1)\\
  c_1(r_2)&\longmapsto c_1(r_1)+\!\!\!\!\!\sum_{\substack{c_H^-\subset C;\\(c,c_H^-)(r_H)\in\HH}}\!\!\!\!\!\delta_{r_H}(c_1,c_H^-)\cdot c_H^-(r_1)
\eea
which gives an isomorphism $A_*$ between $\HM(\MM,r_2)$ and $\HM(\MM,r_1)$.

For clarity, we assume that $[c_1]$ is a homology class in $\HM(\MM,r_1)$ such that $(c_1,c_2)(r_H)\in\HH$ with $\delta_{r_H}(c_1,c_2)=1\in\mathbb{F}$. Thus after passing through a degenerate point $r_H$, the homology class $[c_1]$ changes to $A_*[c_1]=[c_1-c_2]\in\HM(\MM,r_2)$ so called ``bifurcation of a homology class'', see Figure \ref{fig:Handle-slide}. Moreover we note that the spectral value changes along $r$ as below:
$$
\rho([c_1],r_1)=F_3(c_1(r_1))\pf\rho(A_*[c_1]),r_2)=\max\{F_3(c_1(r_2)),F_3(c_2(r_2))\}.
$$
We refer to this change ``transfer of the spectral value''.

\begin{figure}[htb]
\input{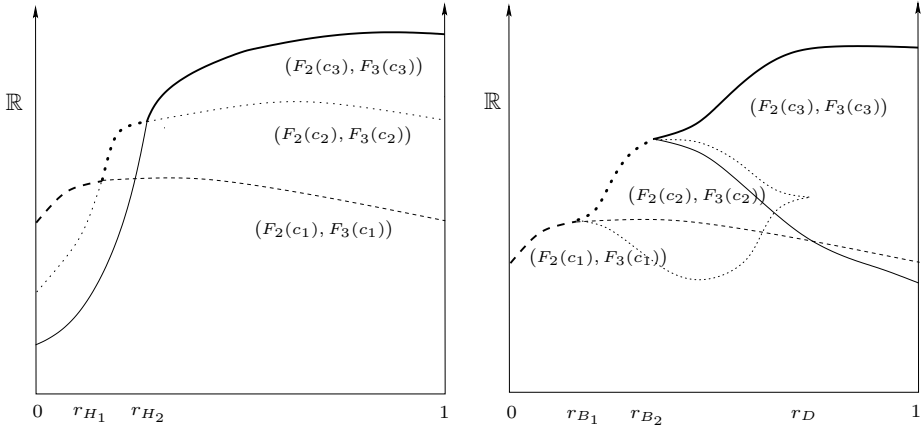}\caption{transfer of the spectral value}
\label{fig:transfer of a spectral value}
\end{figure}

For example, in Figure \ref{fig:transfer of a spectral value} there are two degenerate points $r_{H_1},r_{H_2}\in F_2(\HH)$ and homology class $[c_1]$ changes to $[c_1-c_2]$ and to $[c_1-c_2-c_3]$. Here dashed line, dotted line, and solid line are the front projections of $F(c_1),F(c_2)$, and $F(c_3)$ respectively; moreover bold line indicates critical points which give the spectral value of the homology class $[c_1]$ at each time. In the extremal case that there are infinitely many degeneration points $F_2(\HH)$ and the spectral value $\rho([c_1],\cdot)$ transfers infinitely many times so that it finally diverges to infinity, the homology class $[c_1]$ escapes to infinity and it gives an example described in Theorem A. It is conceivable that infinitely many birth-deaths are also able to cause the escape of a homology class by the analogous argument, see Figure \ref{fig:birth}.
\end{proof}

\begin{Rmk}
We note that the spectral value never transfers at a point in $F_2(\HH)$. To $\delta_{r_H}(c_1,c_H^-)$ be nonzero, $F_3(c_H^-)$ has to be less than $F_3(c_1)$; thus the transfer of the spectral value takes place after a handle-slide. By the same reason, this is true for birth-deaths.
\end{Rmk}

\begin{figure}[htb]
\input{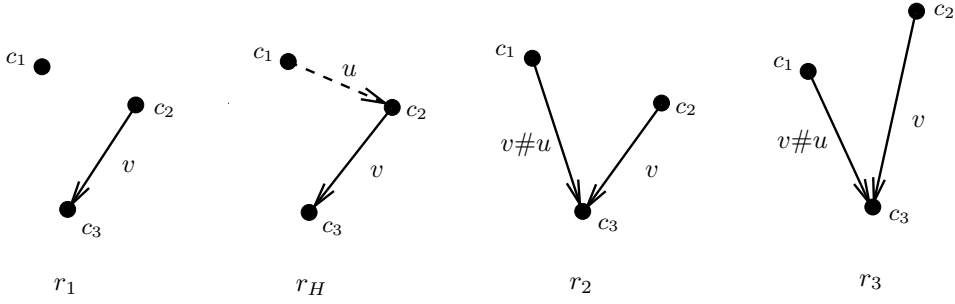}\caption{Handle-slide}
\label{fig:Handle-slide}
\end{figure}
Now we describe the phenomenon illustrated in Figure \ref{fig:Handle-slide}. At time $r_1$, there is only one gradient flow line $v$ between $c_2$ and $c_3$;
it means that $\gamma_{r_1}(c_1,c_2)=\gamma_{r_1}(c_1,c_3)=0$ and $\gamma_{r_1}(c_2,c_3)=1$. At this moment $[c_1]=h\in\HM(\MM,r_1)$ is a nonzero homology class
and the spectral value of $h$ is $\rho(h,r_1)=F_3(c_1(r_1))$. A handle-slide takes place at $r_H$; a degenerate gradient flow line $u$ interchanging $c_1$ and $c_2$ emerges, it means that $(c_1,c_2)(r_H)\in\HH$ and $\delta_{r_H}(c_1,c_2)=1$. After a handle-slide, by gluing two gradient flow lines $v$ and $u$, a new gradient flow line $v\#u$ between $c_1$ and $c_3$ appears at $r_2$, see Axiom $(\gamma3)$; the homology class $[c_1]$ represented by $c_1-c_2$ at $r_2$ but the spectral value is still $F_3(c_1(r_2))$. However after some time, the action value of $c_2$ goes over the action value of $c_1$, thus the spectral value of $h$ at $r_3$ is changed to $\rho(h,r_3)=F_3(c_2(r_3))$.

\begin{figure}[htb]
\input{birth.pstex_t}\caption{Birth}
\label{fig:birth}
\end{figure}

Let above Figure \ref{fig:birth} be graphs of one parameter family of functions $\{f_r\}_{r\in[0,1]}$ on a one dimensional manifold, embedded in a plane; note that these functions are always Morse except at time $r_B\in[0,1]$. At $r_1$, $h=[c_1]$ is a nonzero homology class in $\HM(\MM,r_1)$ and $\rho(h,r_1)=f_{r_1}(c_1)$. At the moment of birth, $r_B\in[0,1]$, a new critical point $c_0$ born which is a degenerate point thus Morse homology cannot be defined at $r_B$. After that, it bifurcates into two critical points $c_B^+$ and $c_B^-$. Now the homology class $h$ is represented by $c_1-c_B^+$; and the spectral value may transfer after some time, in this example $\rho(h,r_2)=f_{r_2}(c_B^+)$. Figure \ref{fig:transfer of a spectral value} illustrates how the spectral value varies along certain homotopies.

\begin{Rmk}\label{Rmk:escape of a critical point}\textbf{(Escape of a critical point.)} In the general Morse theory on noncompact manifolds, critical points may escape to infinity during homotopies and this also violates the invariance property of Morse homology. For instance, let our manifold be $\R^2-l$ where $l:=\{(x,0)\,|\,\frac{1}{2}\leq x\leq1\}$ and an one-parameter family of Morse functions be $f_r(x,y):=(x-r)^2+y^2$. Then $\Crit f_r=\{(r,0)\,|\,r\in[0,\frac{1}{2})\}$, accordingly this critical point escape ever after the time $1/2$, see Figure \ref{fig:escape} below. However in this paper we have assumed that each connected component of $C$ is compact, it means that each critical point of the Morse function stays in a compact region of a manifold during homotopies.
\begin{figure}[htb]
\input{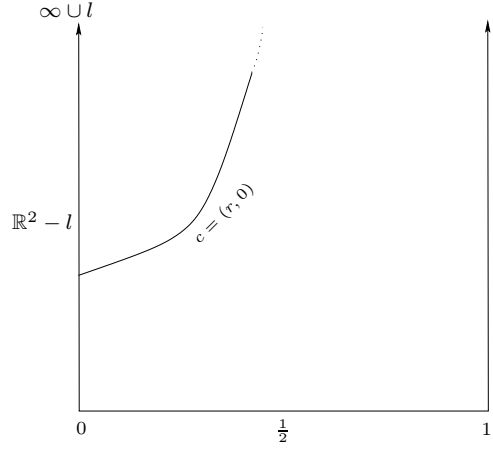}\caption{escape of a critical point}
\label{fig:escape}
\end{figure}
\end{Rmk}

To exclude the escape of homology classes, we need the following hypothesis:
\begin{itemize}
\item[(H1)] There exists a continuous function $\Phi(s):\R\setminus(a,b)\pf\R_{>0}$ such that for all $c\subset C$,
\beqn
\bigg|\frac{\p F_3(c(r))}{\p r}\bigg|\leq\Phi\big(F_3(c(r))\big)\quad \&\quad\int_b^\infty\frac{1}{\Phi(s)}ds=\infty\quad \&\quad\int_{-\infty}^{a}\frac{1}{\Phi(s)}ds=\infty.
\eeq
\end{itemize}
\begin{Rmk}
In the hypothesis (H1), $(a,b)$ can be empty. For example, $\Phi(s)=1/|s|$ satisfies the hypothesis.
\end{Rmk}
\noindent\textbf{Theorem B.} {\em Under the hypothesis \rm{(H1)},} {\em Morse homology is invariant.}

\begin{proof}
Assume on the contrary that for a give homology class $h\in\HM(\MM,0)$, there exist $\{c_\nu\subset C\}_{\nu\in\N}$ components of $C$ and $\{r_\nu\in[0,1]\}_{\nu\in\N}$ such that
\begin{itemize}
\item $r_1=0$, $r_{\nu+1}\geq r_\nu$,
\item $\lim_{\nu\to\infty}\rho(h,r_\nu)=\infty$.
\item the $F_3$-value of $c_\nu$ transfers to $c_{\nu+1}$ at $r_\nu$ and it diverges to infinity; that is, $F_3(c_\nu(r_\nu))=F_3(c_{\nu+1}(r_\nu))$ and $\lim_{\nu\to\infty} F_3(c_\nu(r_\nu))=\infty$.
\item there exists $k\in\N$ such that $F(c_\nu(r_\nu))>b$ for all $\nu\geq k$.
\end{itemize}
Assuming the spectral value goes to the positive infinity, we compute
\bean
\lim_{\substack{n\to\infty\\ n\in\N}}r_n-r_k&=\lim_{n\to\infty}\int_{r_k}^{r_n}1dr\\
&=\lim_{n\to\infty}\sum_{\nu=k}^{n-1}\int_{r_{\nu}}^{r_{\nu+1}}1dr\\
&\geq\lim_{n\to\infty}\sum_{\nu=k}^{n-1}\int_{r_{\nu}}^{r_{\nu+1}}\frac{1}{\Phi(F_3)} \frac{\p F_3(c_{\nu+1}(r))}{\p r}dr\\
&\geq\lim_{n\to\infty}\sum_{\nu=k}^{n-1}\int_{F_3(c_{\nu+1}(r_{\nu}))}^{F_3(c_{\nu+1}(r_{\nu+1}))}\frac{1}{\Phi(s)}ds\\
&=\lim_{n\to\infty}\int_{F_3(c_{k+1}(r_k))}^{F_3(c_n(r_n))}\frac{1}{\Phi(s)}ds=\infty.
\eea
It contradicts to the fact $r_k,r_n\in[0,1]$. On the other hand, if $\lim_{\nu\to\infty}\rho(h,r_\nu)=-\infty$, we may assume that
\begin{itemize}
\item $\lim_{\nu\to\infty} F_3(c_\nu(r_\nu))=-\infty$.
\item there exists $k\in\N$ such that $F(c_\nu(r_\nu))<a$ for all $\nu\geq k$.
\end{itemize}
Then the following similar computation holds.
\bean
\lim_{\substack{n\to\infty\\ n\in\N}}r_n-r_1
&=\lim_{n\to\infty}\sum_{\nu=k}^{n-1}\int_{r_{\nu}}^{r_{\nu+1}}1dr\\
&\geq\lim_{n\to\infty}\sum_{\nu=k}^{n-1}\int_{F_3(c_{\nu+1}(r_{\nu}))}^{F_3(c_{\nu+1}(r_{\nu+1}))}-\frac{1}{\Phi(s)}ds\\
&=\lim_{n\to\infty}\int_{F_3(c_{k+1}(r_k))}^{F_3(c_n(r_n))}-\frac{1}{\Phi(s)}ds=\infty.
\eea
The above two contradictory computation complete the proof of the theorem.
\end{proof}

\begin{Rmk}
In a compact manifold $M$, $\Crit f_r$ for $f_r\in C^\infty(M)$, $r\in[0,1]$ consists of finite components thus (H1) holds with $\Phi\equiv O$ for some constant $O\in\R$ and thus, neither critical points nor homology classes never escape. Therefore, Morse homology on compact manifold is independent of the choice of Morse functions since we can always homotop two Morse functions. We can also find a constant function $\Phi$ in the classical Floer theory. Let $(M,\om)$ be a weakly monotone closed symplectic manifold and $H\in C^\infty(S^1\x M)$ be a time-dependent Hamiltonian function. For each contractible loop $v\in C^\infty(S^1,M)$, we choose a map $w\in C^\infty(D^2,M)$ with $w|_{\p D^2}=v$. With additional equivalences and conditions, the Floer action functional is defined as below (refer to \cite{HS,Sa} for a rigorous framework):
$$
\AA_H(v,w)=-\int_{D^2}w^*\om-\int_0^1H_t(v(t))dt.
$$
Along the homotopy $\{H_r\}_{r\in[0,1]}$ and corresponding critical points $\{(v_r,w_r)\}_{r\in[0,1]}$ we calculate
\bean
\bigg|\frac{\p}{\p r}\AA_{H_r}(v_r,w_r)\bigg|&=\Big|d\AA_{H_r}(v_r,w_r)[\p_r v_r,\p_rw_r]+\p_r\AA_{H_r}(v_r,w_r)\Big|\\
&=\bigg|\int_0^1\p_rH_r(v_r(t))dt\bigg|\\
&\leq||\p_rH_r||_{L^\infty}.
\eea
Thus the value $|\frac{\p}{\p r}\AA_{H_r}(v_r,w_r)|$ is uniformly bounded by some constant for all critical points $(v_r,w_r)$. Accordingly, the invariance property of Floer homology can be proved by the bifurcation method. However for the Rabinowitz action functional, this argument does not hold anymore by the effect of the Lagrange multiplier. Nevertheless if we assume tameness then we get $\Phi$, not necessarily constant, satisfying (H1), see section 4.
\end{Rmk}
\begin{Rmk}
Note that (H1) is a sufficient condition to prevent the escape of a homology class, but not a necessary condition. As an easy example, if we know that there is no intersection points in the Cerf diagram then the transfer of the spectral value never occurs at all; thus every homology class remains along homotopies without any hypothesis. Besides, if we also have an information about the grading of critical points, it is also useful.
\end{Rmk}

On the other hand, if we know the data of the spectral value of a given homology calss at the initial point, we can show that the homology class survives under a mild hypothesis rather than (H1).\\[-1ex]

\begin{itemize}
\item[(H2)] For a given homology class $h\in\HM(\MM,0)$, there exists a continuous function $\Phi_h(s):\R\setminus(a,b)\pf\R_{>0}$ and $\kappa>0$ such that
\beqn
\bigg|\frac{\p F_3(c(r))}{\p r}\bigg|\leq\Phi_h\big(F(c(r))\big)\quad \&\quad\int_{M}^\infty\frac{1}{\Phi_h(s)}ds\geq 1+\kappa\quad \&\quad\int^{m}_{-\infty}\frac{1}{\Phi_h(s)}ds\geq 1+\kappa
\eeq
where $M:=\max\{b,\rho(h,0)\}$, $m:=\min\{a,\rho(h,0)\}$, and $\rho(h,0)$ is the spectral value of $h$ at $0$ defined in \eqref{eq:spectral value}.
\end{itemize}
\begin{Rmk}
In the hypothesis (H2), $(a,b)$ can be empty. For example, $h$ with $|\rho(h,0)|<1$ and $\Phi_h(s)=1/s^2$ satisfy the hypothesis.
\end{Rmk}
\textbf{Theorem C.}  {\em The homology class $h\in\HM(\MM,0)$ satisfying the hypothesis} \rm{(H2)} {\em survives along homotopies.}
\begin{proof}
Similar as the proof of Theorem B, we assume by contradiction that there exists sequences $c_\nu\subset C$ and $r_\nu\in[0,1]$, $\nu\in\N$ with the following properties.
\begin{itemize}
\item $r_1=0$, $r_{\nu+1}>r_\nu$.
\item $\lim_{\nu\to\infty}\rho(h,r_\nu)=\infty$; without loss of generality, we may assume that $\rho(h,r)$ is increasing as $r$ becomes bigger.
\item The action value of $c_\nu$ transfers to $c_{\nu+1}$ at $r_\nu$ and it diverges to infinity; that is, $F_3(c_\nu(r_\nu))=F_3(c_{\nu+1}(r_\nu))$ and $\lim_{\nu\to\infty} F_3(c_\nu(r_\nu))=\infty$.
\item Moreover, if $\rho(h,0)\geq b$, it holds that
$$
\int_{F_3(c_1(r_1))}^\infty\frac{1}{\Phi_h(s)}ds\geq 1+\frac{\kappa}{2}.
$$
If $\rho(h,0)<b$ there exists $r'\in[0,1]$ such that $r_k\leq r'\leq r_{k+1}$ for some $k\in\N$, $F_3(c_k(r'))=b$, and $F_3(c_\nu(r))>b$ for $r'>r$ and $\nu>k$.
\end{itemize}
Then we compute along the same lines as the proof of Theorem B. If $\rho(h,0)\geq b$,
\bean
\lim_{\substack{n\to\infty\\ n\in\N}}r_n-r_1&=\lim_{n\to\infty}\int_{r_1}^{r_n}1dr\\
&=\lim_{n\to\infty}\sum_{\nu=1}^{n-1}\int_{r_{\nu}}^{r_{\nu+1}}1dr\\
&\geq\lim_{n\to\infty}\sum_{\nu=1}^{n-1}\int_{r_{\nu}}^{r_{\nu+1}}\frac{1}{\Phi_h(F_3)} \frac{\p F_3(c_{\nu+1}(r))}{\p r}dr\\
&\geq\lim_{n\to\infty}\sum_{\nu=1}^{n-1}\int_{F_3(c_{\nu+1}(r_{\nu}))}^{F_3(c_{\nu+1}(r_{\nu+1}))}\frac{1}{\Phi_h(s)}ds\\
&=\lim_{n\to\infty}\int_{F_3(c_2(r_1))}^{F_3(c_n(r_n))}\frac{1}{\Phi_h(s)}ds\geq 1+\frac{\kappa}{2}>1.
\eea
It contradicts to $r_n,r_1\in[0,1]$. If $\rho(h,0)<b$, similarly we compute
\bean
\lim_{\substack{n\to\infty\\ n\in\N}}r_n-r'&=\lim_{n\to\infty}\int_{r'}^{r_n}1dr\\
&=\lim_{n\to\infty}\sum_{\nu=k+1}^{n-1}\int_{r_{\nu}}^{r_{\nu+1}}1dr+\int_{r'}^{r_{k+1}}1dt\\
&=\lim_{n\to\infty}\int_{F_3(c_{k}(r'))}^{F_3(c_n(r_n))}\frac{1}{\Phi_h(s)}ds\geq 1+\kappa>1.
\eea
The computation for the case $\lim_{\nu\to\infty}\rho(h,r_\nu)=-\infty$ analogously follows. These contradictory computations conclude the proof.
\end{proof}

\section{Application to Rabinowitz Floer homology}
In this section we discuss the invariance problem of Rabinowitz Floer homology. Very roughly, Rabinowitz Floer homology is a semi-infinite dimensional Morse homology of the Rabinowitz action functional. The invariance problem of Rabinowitz Floer homology is highly nontrivial; it turns out that Rabinowitz Floer homology is invariant under a suitable condition, but a counterexample is not yet known in more general case. First, we recall the notion of Rabinowitz Floer homology and formulate the invariance problems. We refer to \cite{AF} for a brief survey on Rabinowitz Floer homology theory. In fact, the Rabinowitz action functional and Rabinowitz Floer homology have nice properties on restricted contact submanifolds, but they still can be defined on stable manifolds and have significant roles in the magnetic field theory, see \cite{CFP}. In this paper, we focus on stable hypersurfaces, yet our story continues to hold on any stable coisotropic submanifolds; we refer to \cite{Ka1} for Rabinowitz Floer theory on coisotropic submanifolds.

\subsection{Stability and tameness}
In this subsection we briefly recall the notions of stability and tameness; for further discussions, we refer to \cite{CFP,CV} and cited therein.
\begin{Def}\label{def:stable}
A {\em Hamiltonian structure} on a $(2n-1)$-dimensional manifold $\Sigma$  is a closed 2-form $\om\in\Omega^2(\Sigma)$ such that  $\om^{n-1}\ne0$. This Hamiltonian structure is called {\em stable} if there exists a {\em stabilizing 1-form} $\lambda\in\Omega^1(\Sigma)$ such that
\begin{itemize}
\item $\ker\om\subset\ker d\lambda$;
\item $\lambda|_{\ker\om}\ne0$.
\end{itemize}
Furthermore two equations $\lambda(R)=1$ and $i_R\om\ne0$ characterize the unique vector field $R$ on $\Sigma$, so called the {\em Reeb vector field}.
\end{Def}
There are several equivalent formulations of stability.
\begin{Thm}\cite{Wa}
A Hamiltonian structure $(\Sigma,\om)$ is stable if and only if its characteristic foliation is geodesible.
\end{Thm}
\begin{Thm}\cite{Su}
A Hamiltonian structure $(\Sigma,\om)$ is non-stable if and only if there exists a foliation cycle which can be arbitrary well approximated by boundaries of singular 2-chains tangent to the foliation.
\end{Thm}
\begin{Def}
A closed hypersurface $\Sigma$ in a symplectic manifold $(M,\om)$ is called {\em stable} if $\Sigma$ separates $M$ and $\om_{|\Sigma}$ defines a stable Hamiltonian structure. A {\em stable homotopy} is a smooth homotopy $\{(\Sigma_r,\lambda_r)\}_{r\in[0,1]}$ of stable hypersurfaces together with associated stabilizing one forms.
\end{Def}
\begin{Prop}
A closed hypersurface $\Sigma$ in $(M,\om)$ is stable if and only if there exists a tubular neighborhood $\Sigma\x(-\epsilon,\epsilon)\subset M$ of $\Sigma\x\{0\}$ such that $\ker\om|_{\Sigma\x\{r\}}=\ker\om|_{\Sigma\x\{0\}}$.
\end{Prop}
\begin{proof}
A closed 2-form $\om'=\om_{|M}+d(r\lambda)$ endows a sympelctic structure on $\Sigma\x(-\epsilon,\epsilon)$ for enough small $\epsilon>0$. By the Weinstein neighborhood theorem, $\Sigma\x(-\epsilon,\epsilon)$ is symplectomorphic to a neighborhood of $\Sigma$. Conversely a 1-form $\lambda:=i_\frac{\p}{\p r}\om_{|M}$ is a stabilizing 1-form on $\Sigma$.
\end{proof}
\begin{Def}\label{def:tame}
Let $(\Sigma,\lambda)$ be a stable hypersurface in a symplectic manifold $(M,\om)$ being symplectically aspherical, i.e. $\om|_{\pi_2(M)}\equiv0$. We denote by $X(\Sigma)$ the set of closed Reeb orbits in $\Sigma$ which is contractible in $M$. Then we define a function
$\Omega:X(\Sigma)\pf\R$ by
$$
\Omega(v)=\int_{D^2}\bar v^*\om.
$$
where $\bar v$ is any filling disk of $v$, i.e. $\bar v|_{\p D^2}=v$. The symplectically asphericity condition guarantees that the value of this action functional is independent of the choice of filling disks. $(\Sigma,\lambda)$ is called {\em tame} if for all $v\in X(\Sigma)$ there exists a constant $c>0$ satisfying
\beqn
\Bigg|\int_0^1v^*\lambda\Bigg|\leq c\,|\Omega(v)|.
\eeq
A stable homotopy $\{(\Sigma_r,\lambda_r)\}_{r\in[0,1]}$ is said to be {\em tame} if each $(\Sigma_r,\lambda_r)$ is tame with a constant $c>0$ independent of $r\in[0,1]$.
\end{Def}

\begin{Rmk}
There are numerous examples of stable tame or non-tame hypersurfaces in \cite{CFP,CV}.
\end{Rmk}
\begin{Rmk}
If our stable hypersurface $(\Sigma,\lambda)$ is of restricted contact type, that is $\lambda\in\Omega^1(M)$ is a 1-form on $M$ and a primitive of $\om$ on whole $M$, then it is tame with a constant $c=1$.
$$
\Bigr|\int_0^1v^*\lambda\Bigr|=\Bigr|\int_{D^2}\bar v^*\om\Bigr|=|\Omega(v)|.
$$
\end{Rmk}

\subsection{Rabinowitz Floer homology}
Let $(\Sigma,\lambda)$ be a stable hypersurface in a symplectic manifold $(M,\om)$ being symplectically aspherical and $\LLL$ be a component of contractible loop in $C^\infty(S^1,M)$. We choose a {\em defining Hamiltonian function} $H\in C^\infty(M)$ carefully (see \cite{CFP} for details) so that $H^{-1}(0)=\Sigma$, $X_H|_\Sigma=R$, and $X_H$ is compactly supported. Then the Rabinowitz action functional $\AA^H:\LLL\x\R\pf\R$ is defined by
$$
\AA^H(v,\eta):=-\int_{D^2}\bar v^*\om-\eta\int_0^1 H(v(t))dt
$$
where $\bar v$ is a filling disk of $v$. In a restricted contact manifold this Rabinowitz action functional itself gives compactness of gradient flow lines and thus Rabinowitz Floer homology can be defined. But in a stable manifold we need an aid of the auxiliary action functional $\widehat\AA^H:\LLL\x\R\pf\R$
$$
\widehat\AA^H(v,\eta)=-\int_{D^2}\bar v^*d\beta-\eta\int_0^1 H(v(t))dt.
$$
where $\beta$ is a 1-form globally defined on $M$ such that $\beta|_\Sigma=\lambda$, see \cite{CFP} for a rigorous construction of $\beta$.
A critical point of $\AA^\H$, $(v,\eta)\in\Crit\AA^H$, satisfies
\beq\label{eq:critical point equation}
\left\{
  \begin{array}{ll}
    \p_tv(t)=\eta X_H(v(t)), \\[1ex]
    H(v(t))=0.
  \end{array}
\right.
\eeq

It is noteworthy that each critical point $(v,\eta)\in\Crit\AA^H$ gives rise to a closed Reeb orbit with period $\eta$ in the following way: let $v_\eta(t):=v(t/\eta)$ for $t\in\R$, then it is $\eta$-periodic. By the second equation in \eqref{eq:critical point equation}, $v_\eta(t)$ lies in $\Sigma$ and it solves $\p_tv_\eta(t)=X_H(v_\eta(t))=R(v_\eta(t))$.

In addition we observe that for $(v,\eta)\in\Crit\AA^H$,
\bean
\bullet\quad&\bigr|\AA^H(v,\eta)\bigr|=\bigg|\int_{D^2}\bar v^*\om\bigg|=|\Omega(v)|,\\
\bullet\quad&\bigr|\widehat\AA^H(v,\eta)\bigr|=\bigg|\int_0^1 v^*\lambda\bigg|=\bigg|\eta\int_0^1 \lambda(R(v(t)))dt\bigg|=|\eta|.
\eea

Next, we note that $\AA^{H}$ is never Morse because there is a $S^1$-symmetry coming from time-shift on the critical point set. However it is known that $\AA^H$ is generically Morse-Bott (see \cite{CF}), so we are able to compute its Floer homology by choosing an auxiliary Morse function $f$ on a critical manifold $\Crit\AA^H$ and counting gradient flow lines with cascades, refer to \cite{Fr,CF}. Let us set the $\Z/2$-module by
\beq
\CF^{(a,b)}(\AA^H,f):=\Crit^{(a,b)}(f)\otimes\Z/2
\eeq
where
\beq
\Crit^{(a,b)}(f):=\big\{(v,\eta)\in\Crit f\subset\Crit\AA^H\,\big|\,f(v,\eta)\in(a,b)\big\}.
\eeq
Then it becomes a complex with the boundary operator $\p$ defined by counting gradient flow lines with cascades. Then {\em filtered Rabinowitz Floer homology} is defined by
$$
\RFH_a^b(\Sigma,M):=\HF^{(a,b)}(\AA^H)=\H\big(\CF^{(a,b)}(\AA^H),\p\big),
$$
and (full) {\em Rabinowitz Floer homology} is defined by
$$
\RFH(\Sigma,M):=\lim_{\substack{\pf\\b\to\infty}}\lim_{\substack{\longleftarrow\\a\to-\infty}}\RFH^b_a(\Sigma,M).
$$

\subsection{Invariance}

We recall the invariance result of Rabinowitz Floer homology proved by Cieliebak-Frauenfelder-Paternain. They used the continuation method to prove the invariance property and it needed clever but complicated computations. As we mentioned in the introduction, let us believe that the bifurcation method of Rabinowitz Floer theory is worked out:
\begin{itemize}
\item[(H3)] There exists a ``regular homotopy of Floer systems '' in the sense of Lee \cite{Lee1,Lee2} between any two Rabinowitz action functionals.
\end{itemize}
Then we can easily show that there is no escape of homology classes along stable tame homotopies; moreover we can prove the invariance with the relaxed condition rather that tameness.

\begin{Thm}\cite{CFP}\label{thm:CFP}
Assuming \rm{(H3)}, {\em let $\{(\Sigma_r,\lambda_r)\}_{r\in[0,1]}$ be a stable tame homotopy. Then there exist an isomorphism:}
$$
\Psi:\RFH(\Sigma_0,M)\stackrel{\cong}{\pf}\RFH(\Sigma_1,M).
$$
\end{Thm}
We reemphasize that this theorem is proved by \cite{CFP} using the continuation method without (H3).
\begin{proof}
At first we prove that critical points of the Rabinowitz action functional survive during the homotopy. We choose defining Hamiltonian functions $H_r$ for $\Sigma_r$ and $||\p_r H_r||_{L^\infty}<\infty$. We note that if $(v_r,\eta_r)\in\Crit\AA^{H_r}$, then $(v_r,\eta_r)\in\Crit\widehat\AA^{H_r}$ by the stability condition. Using this fact, for $(v_r,\eta_r)\in\Crit\AA^{H_r}$ we compute
\bea
\Big|\frac{\p}{\p r}\eta_r\Big|&=\Big|\frac{\p}{\p r}\widehat\AA^{H_r}(v_r,\eta_r)\Big|\\
&=\Big|d\widehat\AA^{H_r}(v_r,\eta_r)[\p_rv_r,\p_r\eta_r]+\eta_r\int_0^1\p_rH_r(t,v_r(t))dt\Big|\\
&\leq||\p_rH_r||_{L^\infty}|\eta_r|.
\eea
Let $\mathfrak{H}:=\max_{r\in[0,1]}||\p_rH_r||_{L^\infty}$. It directly follows that
$$
|\eta_r|\leq e^\mathfrak{H}|\eta_0|,
$$
thus $\eta_r$ is bounded in terms of the initial value $\eta_0$. From the equation $\p_t v_r(t)=\eta_r X_{H_r}(t,v_r)$, we conclude that a critical point $(v_r,\eta_r)$ does not escape.

In order to show the invariance property for Rabinowitz Floer homology, it remains to show that there is no escape of homology classes. We observe that the tameness implies the hypothesis (H1). We compute
\bean
\Bigg|\frac{\p}{\p r}\AA^{H_r}(v_r,\eta_r)\Bigg|&=\Bigg|\underbrace{d\AA^{H_r}(v_r,\eta_r)}_{=0}[\p_rv_r,\p_r\eta_r]+\eta_r\int_0^1\p_r H_r(t,v_r(t))dt\Bigg|\\
&\leq \mathfrak{H}|\eta_r|\\
&=\mathfrak{H}\big|\widehat\AA^{H_r}(v_r,\eta_r)\big|\\
&\leq c\mathfrak{H}\big|\AA^{H_r}(v_r,\eta_r)\big|.
\eea
With $\Phi(s)=c\mathfrak{H}\cdot |s|$ our hypothesis (H1)
$$
\frac{1}{c\mathfrak{H}}\int_1^\infty\frac{1}{s}ds=\infty\quad\&\quad\frac{1}{c\mathfrak{H}}\int_{-\infty}^{-1}  -\frac{1}{s}ds=\infty
$$
holds and hence Rabinowitz Floer homology is invariant by Theorem B.
\end{proof}

\begin{Def}
We refer to a stable hypersurface $(\Sigma,\lambda)$ as {\em logarithmic-tame} if there exists $c>0$ such that the following holds: For all $v\in X(\Sigma)$,
$$
\Bigg|\int_0^1v^*\lambda\Bigg|\leq c\,|\Omega(v)|\x\log|\Omega(v)|\x\log\log|\Omega(v)|\x\cdots\x\log\cdots\log|\Omega(v)|.
$$
A stable homotopy $\{(\Sigma_r,\lambda_r)\}_{r\in[0,1]}$ is called {\em logarithmic-tame} if each $(\Sigma_r,\lambda_r)$ is logarithmic-tame with a uniform constant $c>0$.
\end{Def}
\begin{Thm}
Rabinowitz Floer homology is invariant along a stable logarithmic-tame homotopy $\{(\Sigma_r,\lambda_r)\}_{r\in[0,1]}$.
\end{Thm}
\begin{proof}
As in the proof of Theorem \ref{thm:CFP}, critical points of the Rabinowitz action functional do not escape. It it enough to show that homology classes also never escape. In this case we take a function
$$
\Phi(s)=c\mathfrak{H}|s|\log |s|\x\log\log |s|\x\cdots\x\log\cdots\log |s|.
$$
It satisfies (H1) and thus Theorem B finishes the proof.
\end{proof}
\begin{Def}
We refer to a stable hypersurface $(\Sigma,\lambda)$ as {\em square-tame} if there exists $c>0$ such that the following holds: For all $v\in X(\Sigma)$,
$$
\Bigg|\int_0^1v^*\lambda\Bigg|\leq c|\Omega(v)|^2.
$$
A stable homotopy $\{(\Sigma_r,\lambda_r)\}_{r\in[0,1]}$ is called {\em square-tame} if each $(\Sigma_r,\lambda_r)$ is square-tame with a uniform constant $c>0$.
\end{Def}
In the square-tame homotopy case, Theorem B works no longer because we have
$$
\int_1^\infty\frac{1}{cs^2}ds=\frac{1}{c}<\infty.
$$
On the other hand, if we know the spectral value (defined in \eqref{eq:spectral value}) of a given homology class and this value is small enough, then the homology class cannot escape during  square-tame homotopies.
\begin{Thm}
Suppose that there exists a constant $\kappa>0$ such that for a square-tame homotopy $\{(\Sigma_r,\lambda_r)\}_{r\in[0,1]}$ with a tame constant $c>0$, the following holds.
$$
\frac{-1}{c+c\kappa}\leq\rho(h,0)\leq \frac{1}{c+c\kappa}\,\,\, \textrm{ for some}\,\,\, h\in\RFH(\Sigma_0,M).
$$
Then the homology class $h\in\RFH(\Sigma_0,M)$ survives along the homotopy.
\end{Thm}
\begin{proof} The hypothesis (H2) holds with $\Phi(s)=1/cs^2$ since
$$
\int_{1/(c+c\kappa)}^\infty\frac{1}{cs^2}\,ds=1+\kappa\quad\&\quad\int_{-\infty}^{{-1/(c+c\kappa)}}\frac{1}{cs^2}ds=1+\kappa.
$$
Therefore Theorem C concludes the proof of the theorem.
\end{proof}
We also can ask if Rabinowitz Floer homology depends on the choice of symplectic forms on $M$. In general, there is no answer yet but as before Rabinowitz Floer homology is invariant with a suitable stability and tameness condition defined below.
\begin{Def}
Let $\lambda_0$ resp. $\lambda_1$ be stabilizing 1-forms on $(M,\Sigma,\om_0)$ resp. $(M,\Sigma,\om_1)$. A smooth homotopy $\{(\om_r,\lambda_r)\}_{r\in[0,1]}$ is called {\em stable} if each $\om_r$ gives symplectic structure on $M$ and $\lambda_r$ is a stabilizing 1-form on $(M,\Sigma,\om_r)$.
\end{Def}
To define the tameness condition and the Rabinowitz action functional, we assume that each $(M,\om_r)$ is symplectically aspherical.
\begin{Def}
A stable homotopy $\{(\om_r,\lambda_r)\}_{r\in[0,1]}$ is said to be {\em tame} if there exists a constant $c>0$ independent of $r\in[0,1]$ such that
$$
\bigg|\int_0^1 v^*\lambda_r\bigg|\leq c|\Omega_r(v)|,\quad v\in X(\Sigma)
$$
where $\Omega_r(v)=\int_{D^2}\bar v^*\om_r$ for $r\in[0,1]$. Instead of the above formula, if it holds that
$$
\bigg|\int_0^1 v^*\lambda_r\bigg|\leq c|\Omega_r(v)|\x\log|\Omega_r(v)|\x\log\log|\Omega_r(v)|\x\cdots\x\log\cdots\log|\Omega_r(v)|,
$$
then we say that a stable homotopy $\{(\om_r,\lambda_r)\}_{r\in[0,1]}$ is {\em logarithmic-tame}.
\end{Def}

Let us indicate the dependency of the symplectic structure on the Rabinowitz action functional and Rabinowitz Floer homology in the following way. We define the Rabinowitz action functional on $(M,\om_r)$ by
$$
\AA^\H_{\om_r}(v,\eta)=-\int_{D^2}\bar v^*\om_r-\eta\int_0^1H(v(t))dt.
$$
With this action funtional, we can define Rabinowitz Floer homology $\RFH(\Sigma,M,\om_r)$ for a stable hypersurface $(\Sigma,\lambda_r)$ in $(M,\om_r)$ as before.
\begin{Thm}
Let $\{(\om_r,\lambda_r)\}_{r\in[0,1]}$ be a stable and logarithmic-tame homotopy. If assuming \rm{(H3)}, we have
$$
\RFH(\Sigma,M,\om_0)\cong\RFH(\Sigma,M,\om_1).
$$
\end{Thm}
\begin{proof}

As before, we define an auxiliary action functional $\widehat\AA^H_{\lambda_r}:\LLL\x\R\pf\R$ by
$$
\widehat\AA^H_{\lambda_r}(v,\eta)=-\int_{D^2}\bar v^*d\beta_r-\eta\int_0^1H(t,v(t))dt
$$
where $\beta_r\in\Omega^1(M)$ is an extension of $\lambda_r\in\Omega^1(\Sigma)$, i.e. $\beta_r|_\Sigma=\lambda_r$, see \cite{CFP} for a rigorous construction of $\beta_r$. For $(v_r,\eta_r)\in\Crit\AA^H_{\om_r}$, it holds that
\beq
\bigg|\frac{\p}{\p r}\eta_r\bigg|=\bigg|\frac{\p}{\p r}\widehat\AA^H_{\lambda_r}(v_r,\eta_r)\bigg|=\bigg|\int_{S^1}v_r^*\dot\lambda_r\bigg|=\mathfrak{R}|\eta_r|.
\eeq
where $\mathfrak{R}:=\max_{r\in[0,1]}||\dot\lambda_r(R_r)||_{L^\infty(\Sigma)}$, $R_r$ is the Reeb vector field with respect to $\lambda_r$. As in the proof of Theorem \ref{thm:CFP}, the above computation yields that critical points do not escape. Next we show the survival of homology classes. We consider the universal cover $(\widetilde M,\widetilde\om_r)$ of $M$ where $\widetilde\om_r$ is the lift of $\om_r$. We choose a compatible almost complex structure $J_r$ on $(M,\om_r)$ so that $g_r(\cdot,\cdot):=\om_r(\cdot,J_r\cdot)$ is a Riemannian metric on $M$. Then we lift $g_r$ to $\widetilde M$, say $\tilde g_r$. Let $\widetilde \Sigma_{\star}(\cong\Sigma)$ be one of the fundamental domains in $\widetilde\Sigma\subset\widetilde M$ and $\tilde v_r:S^1\pf \widetilde M$ intersecting $\Sigma_{\star}$ be the lift of $v_r$. Since we have assumed the symplectical asphericity of $(M,\om_r)$, there exists a primitive 1-form $\sigma$ of $\dot{\widetilde\om}_r$. We let
$$
\mathfrak{S}:=\max_{x\in \tilde v(S^1)}\{||\sigma(x)||_{\tilde g}\,|\,(v,\eta)\in\Crit\AA^H\}.
$$
We define an equivalence relation such that $(v,\eta)\sim (v_0,\eta_0)$ if $\big(v(t),\eta\big)=\big(v_0(nt),n\eta_0\big)$ for some $2\leq n\in\N$ or $v(t)=v(t+r)$ for some $r\in S^1$. We note that there are only finitely many nonconstant representative classes and we can lift $v$ and $v_0$ with $(v,\eta)\sim(v_0,\eta_0)$ so that $\tilde v(S^1)=\tilde v_0(S^1)$. Thus $\mathfrak{S}$ has finite value since $\bigcup_{(v,\eta)\in\Crit\AA^H}\tilde v(S^1)$ is compact. Now we compute
\bean
\bigg|\frac{\p}{\p r}\AA^H_{\om_r}(v_r,\eta_r)\bigg|&=\bigg|\int_{D^2}\bar v_r^*(\dot\om_r)\bigg|=\bigg|\int_{D^2}\bar{\tilde v}_r^*\dot{\widetilde\om}_r\bigg|
=\bigg|\int_{S^1}\tilde{ v}_r^*\sigma\bigg|\\
&\leq\mathfrak{S}\int_{S^1}||\p_t\tilde v_r||_{\tilde g}dt\\
&=\mathfrak{S}\int_{S^1}||\p_t v_r||_{g}dt\\
&\leq\mathfrak{S}|\eta_r|\int_{S^1}||X_H(v_r)||_{g}dt\\
&\leq \Theta|\eta_r|\\
&=\Theta\bigr|\widehat\AA^H_{\lambda_r}(v_r,\eta_r)\bigr|\\
&\leq c\Theta\bigr|\AA^H_{\om_r}(v_r,\eta_r)\bigr|\x\log\bigr|\AA^H_{\om_r}(v_r,\eta_r)\bigr|\x\cdots\x\log\cdots\log\bigr|\AA^H_{\om_r}(v_r,\eta_r)\bigr|.
\eea
where $\Theta=\mathfrak{S}||{X_H}|_\Sigma||_{L^\infty}$ and $c$ is the tame constant. This computation shows that a stable and logarithmic-tame homotopy satisfies the hypothesis (H1) with the function
$$
\Phi(s)=c\Theta|s|\log |s|\x\log\log |s|\x\cdots\x\log\cdots\log |s|
$$
and then Theorem B concludes the proof.
\end{proof}
\begin{Rmk}
We expect that the previous theorem also can be proved by the continuation method without assuming (H3); we refer to \cite{BF} for the continuation method in the virtually contact case. Without doubt, our arguments are also valid in the virtually contact case.
\end{Rmk}

\section{Appendix: Legendrian and pre-Lagrangian}
In this appendix, we briefly recall a part of the contact geometry, Legendrian curves and pre-Lagrangian submanifolds; we refer to \cite{EHS,Ge} for the deeper and wider concepts.
\begin{Def}
Let $M$ be a manifold of dimension $2n+1$. A contact structure on $M$ is a maximally non-integrable hyperplane field $\xi=\ker\alpha\subset TM$, $\alpha\in\Omega^1(M)$, i.e. $\alpha\wedge(d\alpha)^n\neq0$. Such a 1-form $\alpha$ is called a {\em contact form} and the pair $(M,\xi)$ is called a {\em contact manifold}.
\end{Def}
A defining 1-form $\alpha$ is unique up to nowhere vanishing functions, that is, $\ker\alpha=\ker f\alpha$ for any nowhere vanishing function $f$ on $M$.
Let $\mathcal{S}(M,\xi)$ be the trivial subbundle of $T^*M$ whose fiber over $q\in M$ consists of all non-zero linear forms annihilating $\xi_q\subset T_qM$ and defining its coorientation. The bundle $\mathcal{S}(M,\xi)$ is a principal $\R$-bundle with the $\R$-action:
$$
r\cdot \Theta=e^r\Theta,\quad r\in\R,\,\,\Theta\in\mathcal{S}(M,\xi).
$$
Furthermore the canonical 1-form $\lambda=pdq$ on $T^*M$ gives a symplectic structure $d\lambda|_{\mathcal{S}(M,\xi)}$ on $\mathcal{S}(M,\xi)$. The symplectic manifold
$$(\mathcal{S}(M,\xi),d\lambda|_{\mathcal{S}(M,\xi)})$$
is called a {\em symplectization} of $(M,\xi)$. We note that a section of the bundle $\pi:\mathcal{S}(M,\xi)\pf M$ is a contact form.

\begin{Def}
An $(n+1)$-dimensional submanifold $L$ of a contact manifold $(M,\xi)$ satisfying the following two properties, is called {\em pre-Lagrangian}.
\begin{itemize}
\item[(i)] $L$ is transverse to $\xi$,
\item[(ii)] The distribution $\xi\cap TL$ is integrable and can be defined by a closed 1-form.
\end{itemize}
\end{Def}
The motivation of the notion of ``pre-Lagrangian" is provided the following proposition.
\begin{Prop}\cite{EHS}
If $L$ is a pre-Lagrangian submanifold in $(M,\xi)$ then there exists a Lagrangian submanifold $\widetilde L$ in the symplectic manifold $\mathrm{Sympl}(M,\xi)$ such that $\pi(\widetilde L)=L$. The cohomology class $\lambda\in\H^1(L;\R)$ such that $\pi^*\lambda=[\alpha|_{\widetilde L}]$ is defined uniquely up to multiplication by a non-zero constant. Conversely if $\widetilde L\subset \mathcal{S}(M,\xi)$ is a Lagrangian submanifold then $\pi(\widetilde L)=L$ is a pre-Lagrangian in $M$.
\end{Prop}
Thus a pre-Lagrangian submanifold carries a canonical projective class of the form $\lambda$. By definition there exists a contact form $\beta$ with $d\beta|_L=0$; in fact the desired Lagrangian submanifold $\widetilde L$ is the graph of $\beta|_L$.

\begin{Def}
A {\em Legendrian knot} in a contact 3-manifold $(M,\xi)$ is a Legendrian embedding $\gamma:S^1\pf M$, i.e. $\gamma'(\theta)\in\xi_{\gamma(\theta)}$ for all $\theta\in S^1$. A {\em Legendrian chord} is a Legendrian embedding $\gamma:[a,b]\pf M$ which begins and ends on pre-Lagrangian submanifolds.
\end{Def}
In this appendix, we consider $\R^2\x [0,1]$ with standard contact structure $\xi_\st=\ker\alpha_\st$ where $\alpha_\st=dz+xdy$ for $(x,y,z)\in\R\x[0,1]\x\R$. Let $\gamma$ be either a Legendrian knot or a Legnedrian chord in $\R^2\x[0,1]$ and write $\gamma(s)=(x(s),y(s),z(s))$. Then the Legendrian condition yields
$$
\alpha_\st(\gamma')=z'+xy'\equiv0.
$$
\begin{Def}
The {\em front projection} of a curve $\gamma(s)=(x(s),y(s),z(s))$ in $\R\x[0,1]\x\R$ is a curve
$$
\gamma_\FF(s)=(y(s),z(s))\subset [0,1]\x\R.
$$
\end{Def}
If a curve $\gamma$ is  Legendrian then $y'(s)=0$ implies $z'(s)=0$, thus the front projection of $\gamma$ has singular points where $y'=0$, so called {\em cusp points}; moreover it does not have vertical tangencies. We call $\gamma$ or $\gamma_\FF$ {\em generic} if cusp points are isolated.
\begin{Lemma}
Let $\gamma_\FF:(a,b)\pf[0,1]\x\R$ be a front projection of a certain Legendrian immersion. Then away from the cusp points we can recover the unique Legendrian immersion $\gamma:(a,b)\pf(\R\x[0,1]\x\R,\xi_\st)$ via
$$
x(s)=-\frac{z'(s)}{y'(s)}.
$$
The curve is embedded if and only if $\gamma_\FF$ has only transverse self-intersections.
\end{Lemma}
\begin{Rmk}\label{rmk:legendrian lifting}
Consider one-parameter family of the functions $\{f_r\}_{r\in[0,1]}$ on a manifold $M$, let $(r,f_r(x_r))\in [0,1]\x\R$ be one-parameter family of the critical values of $f_r$ where $x_r\in\Crit f_r$. We parameterize this one-dimensional space $(r(s),f_{r(s)}(x_r(s)))$ and compute
\bean
\frac{d}{ds}\bigr(r(s),f_{r(s)}(x_r(s))\bigr)&=\bigr(\dot r(s),df_{r(s)}(x_r(s))[\dot x_r(s)]+\dot r(s)\dot f_{r(s)}(x_r(s))\bigr)\\
&=\bigr(\dot r(s),\dot r(s)\dot f_{r(s)}(x_r(s))\bigr).
\eea
It is known that there exists a homotopy of two Morse function (the Floer action functional case is proved by Lee \cite{Lee1,Lee2}) so that the curve $(r(s),f_{r(s)}(x_r(s)))$ is generic. Moreover we also may assume that this curve has only transverse self-intersections; thus this curve can be lifted a unique Legendrian curve or chord in $\R\x[0,1]\x\R$.
\end{Rmk}


\begin{thebibliography}{99}
\bibitem[AF]{AF} P. Albers, U. Frauenfelder,
{\em Rabinowitz Floer homology : A survey}, (2010), arXiv:1001.4272.
\bibitem[BF]{BF} Y. Bae, U. Frauenfelder, {\em Continuation homomorphism in Rabinowitz Floer homology for symplectic deformations}, (2010), arXiv:1010.1649.
\bibitem[CF]{CF} K. Cieliebak, U. Frauenfelder, {\em Morse homology on noncompact manifolds,} (2009), arXiv:0911.1805.
\bibitem[CFP]{CFP} K. Cieliebak, U. Frauenfelder, G.P. Paternain, {\em Symplectic topology of Ma\~n\'e's critical values,} (2009), arXiv:0903.0700, to appear in Geometry \& Topology.
\bibitem[Co] {Co} A. Cotton-Clay, {\em Symplectic Floer homology of area-preserving surface diffeomorphisms}, Geometry $\&$ Topology, \textbf{13} (2009), 2619--2674.
\bibitem[CV]{CV} K. Cieliebak, E, Volkov, {\em First steps in stable Hamiltonian topology}, (2010), arXiv:1003.5084.
\bibitem[EHS]{EHS} Y. Eliashberg, H. Hofer, D. Salamon, {\em Lagrangian intersections in contact geomtry}, GAFA, Vol. 5, No. 2 (1995).
\bibitem[Fl1]{Fl1} A. Floer, {\em Morse theory for lagrangian intersections,} J. Differential Geometry, \textbf{28} (1988), 513--547.
\bibitem[Fl2]{Fl2} A. Floer, {\em Symplectic fixed points and holomorphic spheres}, Comm. Math. Phys. \textbf{120} (1989), no.4, 575--611.
\bibitem[Fr]{Fr} U. Frauenfelder, {\em The Arnold-Givental conjecture and moment Floer homology,} IMRN. \textbf{42} (2004), 2179--2269.
\bibitem[Ge]{Ge} H. Geiges, {\em An introduction to contact topology}, Cambridge Studies in Advanced Mathematics, vol. 109, Cambridge University Press, Cambridge, 2008.
\bibitem[Hu]{Hu} M. Hutchings, {\em Reidemeister torsion in generalized Morse theory}, Forum Math. \textbf{14} (2002), no.2, 209--244.
\bibitem[HL1]{HL1} M. Hutchings, Y-J. Lee, {\em Circle-valued Morse theory, Reidemeister torsion, and Seiberg-Witten invariants of three manifolds}, Topology \textbf{38} (1999), no. 4, 861--888.
\bibitem[HL2]{HL2} M. Hutchings, Y-J. Lee, {\em Circle-valued Morse theory and Reidemeister torsion}, Geometry and Topology \textbf{3} (1999), no. 4, 369--396.
\bibitem[HS]{HS} H. Hofer, D. Salamon, {\em Floer homology and Novikov rings}, The Floer Memorial Volume, edited by H. Hofer, C. Taubes, A. Weinstein, and E. Zehnder, Birkhauser 1995, pp 483-524.
\bibitem[Ka1]{Ka1} J. Kang, {\em Generalized Rabinowitz Floer theory and coisotropic intersections,} (2010), arXiv:1003.1009.
\bibitem[Ka2]{Ka2} J. Kang, {\em Survival of infinitely many critical points for the Rabinowitz action functional}, (2010), arXiv:1006.2686, to appear in J. Modern dynamics.
\bibitem[Lee1]{Lee1} Y.-J. Lee, {\em Reidemeister Torsion in Floer-Novikov Theory and counting pseudo-holomorphic tori, I.}, J. Symplectic Geometry \textbf{3} (2005), no. 2, 221--311.
\bibitem[Lee2]{Lee2} Y.-J. Lee, {\em Reidemeister Torsion in Floer-Novikov Theory and counting pseudo-holomorphic tori, II.}, J. Symplectic Geometry
\bibitem[Us]{Us} M. Usher, {\em Vortices and a TQFT for Lefschetz fibrations on 4-manifolds,} Algebr. Geom. Topol. \textbf{6} (2006), 1677--1743.
\bibitem[Oh]{Oh} Y.-G. Oh, {\em Floer mini-max theory, the Cerf diagram, and the spectral invariants}, J. Korean Math. Soc. \textbf{46} (2009), No.2, 363--447.
\bibitem[Sa]{Sa}
D.A. Salamon, {\em Lectures on Floer homology,} in
"Symplectic Geometry and Topology",  Eds: Y. Eliashberg and
L. Traynor, IAS/Park City Mathematics series, \textbf{7}, (1999), 143--230.
\bibitem[Su]{Su} D. Sullivan, {\em A foliation of geodesics is chracterized by having no tangent homologies}, J. Pure and Appl. Algebra \textbf{13} (1978), 101--104.
\bibitem[Wa]{Wa} A. Wadsley, {\em Geodesic foliations by circles}, J. Diff. Geom. \textbf{10} (1975), 541--549.
\end{thebibliography}
\end{document}